\documentclass[reqno]{amsart}
\usepackage{amsthm,graphicx,enumitem,tikz}
\usepackage{hyperref}
\usepackage{xcolor}

\newtheorem{theorem}{Theorem}[section]
\newtheorem{lemma}{Lemma}[section]

\theoremstyle{remark}
\newtheorem{remark}{Remark}[section]
\theoremstyle{definition}

\usepackage[ruled,linesnumbered]{algorithm2e}[section]
\SetKwInOut{Constants}{Constants}

\usepackage{pifont}

\title[Fast expansion into harmonics on the disk]{Fast expansion
into harmonics on the disk:\\ a steerable basis with fast radial convolutions}

\author[N.F. Marshall]{Nicholas F. Marshall}
\email{marsnich@oregonstate.edu}
\author[O. Mickelin]{Oscar Mickelin}
\email{hm6655@princeton.edu}
\author[A. Singer]{Amit Singer}
\email{amits@math.princeton.edu}
\thanks{N.F.M. was supported in
part by NSF DMS-1903015. A.S. was supported in part by AFOSR FA9550-20-1-0266, the Simons Foundation Math+X
Investigator Award,  NSF BIGDATA Award IIS-1837992, NSF DMS-2009753, and NIH/NIGMS
1R01GM136780-01. }
\begin{document}

\maketitle

\begin{abstract}
We present a fast and numerically accurate method 
for expanding digitized $L \times L$ images representing functions on $[-1,1]^2$
supported on the disk $\{x \in \mathbb{R}^2 : |x|<1\}$ in the harmonics (Dirichlet
Laplacian eigenfunctions) on the disk. Our method, which we refer to as the Fast Disk Harmonics Transform (FDHT), runs in $\mathcal{O}(L^2 \log L)$ operations. This basis
is also known as the Fourier-Bessel basis, and it has several computational
advantages: it is orthogonal, ordered by frequency, and
steerable in the sense that images expanded in the basis can be rotated by
applying a diagonal transform to the coefficients. Moreover, we show that
convolution with radial functions can also be efficiently computed by applying a diagonal transform to the coefficients.
\end{abstract}

\section{Introduction}

\subsection{Motivation}
Decomposing a function into its Fourier series can be viewed as representing a
function in the eigenfunctions of the Laplacian on the torus $\mathbb{T}:=
[0,2\pi]$ where $0$ and $2\pi$ are identified. Indeed,
$$
    -\Delta e^{\imath kx} = k^2e^{\imath kx}.
$$
The eigenfunctions of the Laplacian (harmonics) on the disk
$\{x \in \mathbb{R}^2: |x| < 1\}$ that satisfy the Dirichlet boundary conditions can be written in polar
coordinates $(r,\theta) \in [0,1) \times [0,2\pi)$ as
\begin{equation} \label{eq:eigenfun}
\psi_{n k}(r,\theta) = c_{n k} J_n ( \lambda_{n k} r) e^{\imath n \theta} ,
\end{equation}
where $c_{n k}$ is a normalization constant, $J_n$ is the $n$-th order Bessel function of the first kind, and
$\lambda_{n k}$ is the $k$-th smallest positive root of $J_n$. The indices run over
$(n,k) \in \mathbb{Z} \times \mathbb{Z}_{> 0}$. The functions $\psi_{nk}$ satisfy
\begin{equation} \label{eq:eigenval}
    -\Delta \psi_{nk} =  \lambda_{n k}^2 \psi_{nk}.
\end{equation}
In this paper, we present
a fast and accurate transform of digitized $L \times L$ images into this eigenfunction basis often referred to as the Fourier-Bessel basis. 
For computational purposes, this basis is
convenient for a number of reasons: 
\begin{enumerate}[label={(\roman*)}]
\item \emph{Orthonormal:} these eigenfunctions are an orthonormal basis for square integrable functions
on the disk.
\item \emph{Ordered by frequency:} the basis functions are ordered by eigenvalues, which can be interpreted as frequencies
due to the connection with the Laplacian and Fourier series described above. Low-pass filtering can be performed by retaining basis coefficients up to a given threshold.
\item \emph{Steerable:} functions expanded in the basis can be rotated by
applying a diagonal transform corresponding to phase modulation of the coefficients.
\item \emph{Fast radial convolutions:} we show that the convolution with radial
functions  can be computed by applying a diagonal transform to the coefficients.
\end{enumerate}

\begin{figure}[t!]
\centering
\begin{tikzpicture}[scale=\textwidth/9cm]
\node[anchor=north west] at (0,2.5) {\includegraphics[width=3cm]{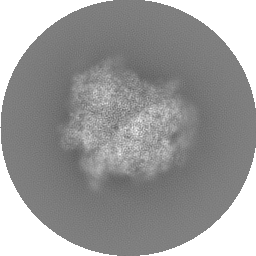}};
\node[anchor=north west] at (0,0) {\includegraphics[width=3cm]{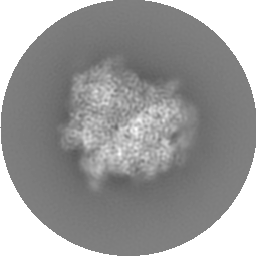}};
\node[anchor=north west] at (2.5,2.5) {\includegraphics[width=3cm]{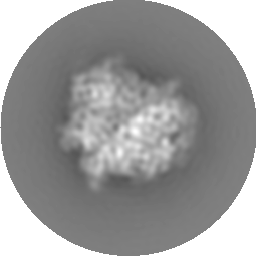}};
\node[anchor=north west] at (2.5,0) {\includegraphics[width=3cm]{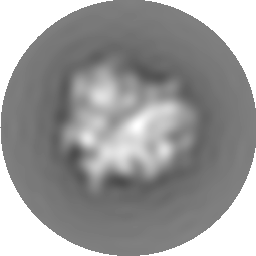}};
\node[anchor=north west] at (5.5,2.5) {\includegraphics[width=3cm]{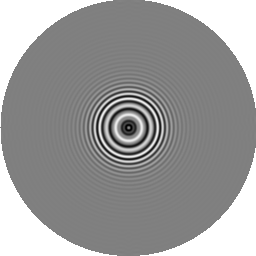}};
\node[anchor=north west] at (5.5,0) {\includegraphics[width=3cm]{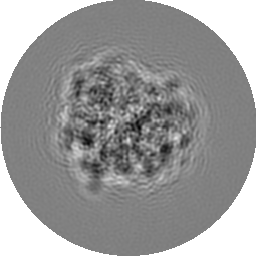}};
\node at (0,2.5) {\LARGE a};
\node at (0,0) {\LARGE b};
\node at (2.5,2.5) {\LARGE c};
\node at (2.5,0) {\LARGE d};
\node at (5.5,2.5) {\LARGE e};
\node at (5.5,0) {\LARGE f};
\end{tikzpicture}

\caption{ Illustration of our method for $L \times L$ images with $L=256$.
Original image (a), a low-pass filter of the original image using a decreasing
number of basis functions (b--d), radial
function (e), convolution of the original image with radial
function (f).}  \label{fig:01}
\end{figure}
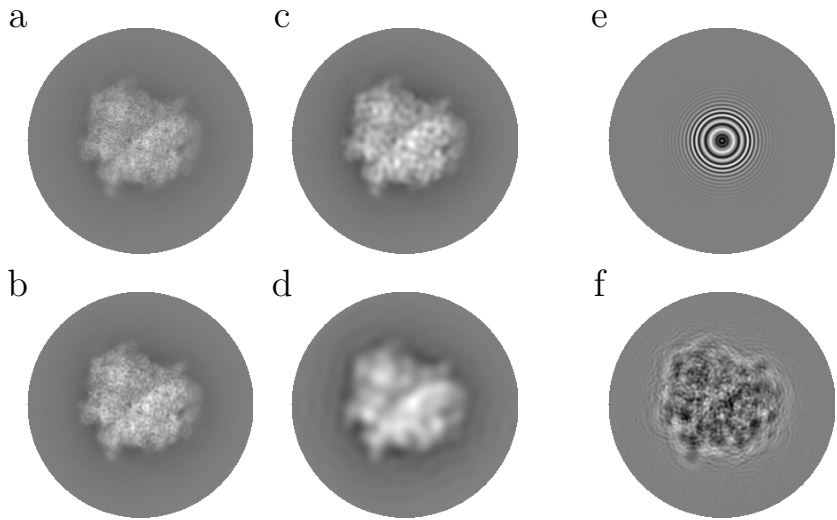
Our FDHT method involves $\mathcal{O}(L^2\log L)$ operations and has precise accuracy guarantees. Python
code that implements our FDHT method is publicly available online\footnote{An implementation is available at \url{https://github.com/nmarshallf/fle_2d}.}. To the best of our knowledge, existing methods for computing the expansion coefficients in a steerable basis   \cite{MR3472531,Zhao2014,landa2017approximation,landa2017steerable} either have computational complexity $\mathcal{O}(L^3)$ or suffer from low numerical precision, see 
\S \ref{sec:past_work} for a more detailed discussion of past work.

Steerable bases have been utilized in numerous image-processing problems
including image alignment \cite{rangan2020factorization}, image classification
\cite{Zhao2014} and image denoising \cite{MR3472531}, including applications to
machine learning \cite{cheng2018rotdcf,cohen2016steerable,weiler2018learning} and data-driven science, such as  applications
to cryo-electron microscopy (cryo-EM) \cite{cheng2015primer,nogales2015cryo},
and computer vision  \cite{papakostas2007new}, among other areas.

There are many possible choices of steerable bases, for instance, Slepian functions (also known as 2-D prolate spheroidal wave functions)
\cite{landa2017approximation,landa2017steerable,slepian1961prolate}, or Zernike polynomials which are widely used in optics \cite{von1934beugungstheorie}.
The harmonics on the disk
(which satisfy Dirichlet boundary conditions) \cite{MR3472531,Zhao2014} are one natural choice due to their
orthogonality, ordering by frequency,  and fast radial convolution.

We illustrate the frequency ordering property of the Laplacian eigenbasis by
performing a  low-pass
filter by projecting onto the span of eigenfunctions whose eigenvalues are
below a sequence of bandlimits that decrease the number of basis functions
successively by 
factors of four, starting from $39593$ coefficients. Since the basis is orthonormal, this is equivalent to setting
coefficients above the bandlimit equal to zero; see Fig. \ref{fig:01}
(a--d). Further, we demonstrate the radial convolution property by illustrating the convolution
with a point spread function, which is a function used in computational microscopy \cite{wade1992brief}; see Fig. \ref{fig:01} (e--f).
The image used for this example is a tomographic projection of a 3-D density map representing a  bio-molecule (E. coli 70S ribosome) \cite{shaikh2008spider}.

\subsection{Notation} \label{sec:notation}
We denote the $L^q$-norm of a function $g: \mathbb{R}^2 \rightarrow \mathbb{C}$ and the $\ell^q$-norm of a vector $v \in \mathbb{C}^d$ by $ \|g\|_{L^q} := (\int_{\mathbb{R}^2} |g(x)|^q dx)^{1/q}$ and  $\|v\|_{\ell^q} := (\sum_{j=1}^d |v_j|^q )^{1/q}$,
respectively.

Let $f$ be an $L \times L$ image whose pixel values $f_{j_1 j_2}$ are samples of a function $\tilde{f} : [-1,1]^2 \rightarrow \mathbb{R}$ that is supported on the unit disk $\{ x \in
\mathbb{R}^2 : |x| < 1 \}$.  More precisely, we  define the pixel locations by
\begin{equation} \label{eq:pixel_locs}
x_{j_1 j_2} := (h j_1- 1, h j_2- 1),
\quad \text{where} \quad h := 1/\lfloor (L+1)/2 \rfloor,
\end{equation}
 and assume the pixel values satisfy $f_{j_1 j_2} = \tilde{f}(x_{j_1 j_2}).$ Let
\begin{equation} \label{eq:enum_p}
x_1,\ldots,x_p \quad \text{and} \quad f_1,\ldots,f_p
\end{equation}
denote an enumeration of the pixel locations and corresponding pixel values, respectively, where $p = L^2$ is the number of pixels in the image. For any given bandlimit $\lambda > 0$, let 
\begin{equation} \label{eq:m}
m = \{ (n,k) \in \mathbb{Z} \times \mathbb{Z}_{>0} : \lambda_{nk} \le \lambda\}
\end{equation} 
denote the number of Bessel function roots (square root of eigenvalues, see 
\eqref{eq:eigenval}) below the bandlimit $\lambda$, and let
\begin{equation} \label{eq:enum_m}
\lambda_1 \le \cdots \le \lambda_m \quad \text{and} \quad \psi_1,\ldots,\psi_m
\end{equation}
denote an enumeration of the Bessel function roots below the bandlimit and corresponding eigenfunctions, respectively. Let
$$
n_1,\ldots,n_m, \quad \text{and} \quad k_1,\ldots,k_m,
$$
be enumerations such that $\psi_{n_j k_j} = \psi_j$. In the following, we switch between using single subscript notation ($x_j$, $f_j$, $\lambda_j$, $\psi_j$) and double subscript notation ($x_{j_1 j_2}$, $f_{j_1 j_2}$, $\lambda_{nk}$, $\psi_{nk})$ depending on which is more convenient; the choice will be clear from the context.

\subsection{Main result} \label{sec:main_result}
We consider the linear transform
$B : \mathbb{C}^m \rightarrow \mathbb{C}^p$ which
maps coefficients to images by 
\begin{equation} \label{eq:operator_B}
(B \alpha)_j = \sum_{i=1}^m \alpha_i \psi_i(x_j) h,
\end{equation}
and its adjoint
transform $B^* : \mathbb{C}^p \rightarrow \mathbb{C}^m$ which maps images to
coefficients by
\begin{equation} \label{eq:operator_B^*}
(B^* f)_i =\sum_{j=1}^p f_j \overline{\psi_i(x_j)} h,
\end{equation}
where the normalization constant  $h$ is included so that $B^* B$ is close to the identity. To provide intuition about $B$ and  $B^*$, we visualize some of the basis functions $\psi_i$ in Figure \ref{fignew}.
The main result of this paper can be informally stated as follows.

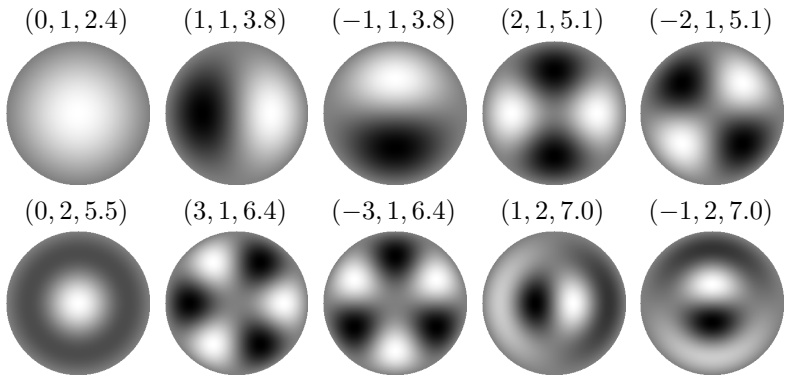
\begin{figure}[t!]
\centering
\begin{tikzpicture}[scale=\textwidth/9cm]
\node[anchor=north west] at (0,1.5) {\includegraphics[width=.15\textwidth]{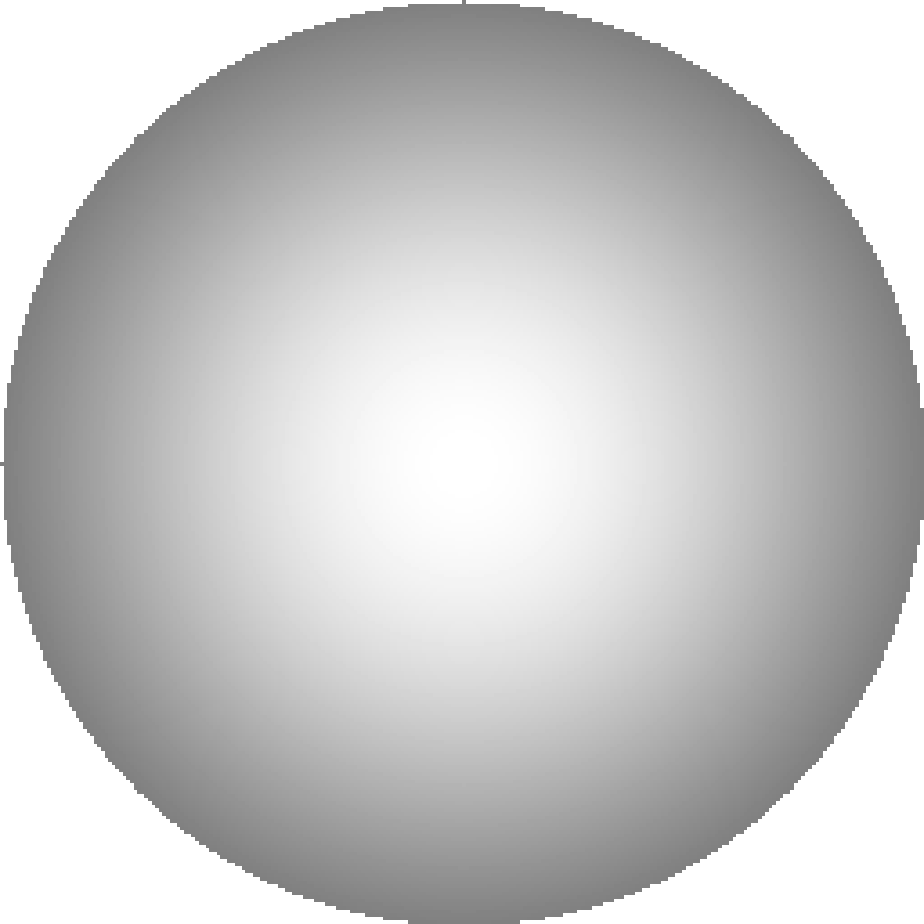}};
\node[anchor=north west] at (1.5,1.5) {\includegraphics[width=.15\textwidth]{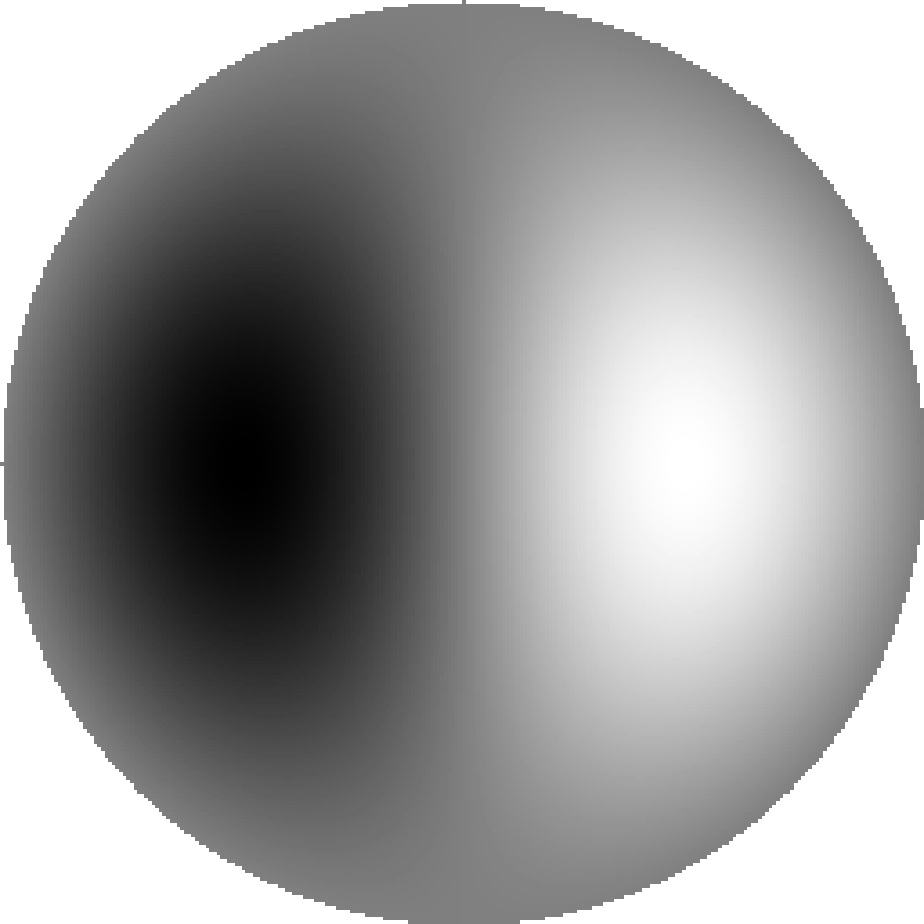}};
\node[anchor=north west] at (3,1.5) {\includegraphics[width=.15\textwidth]{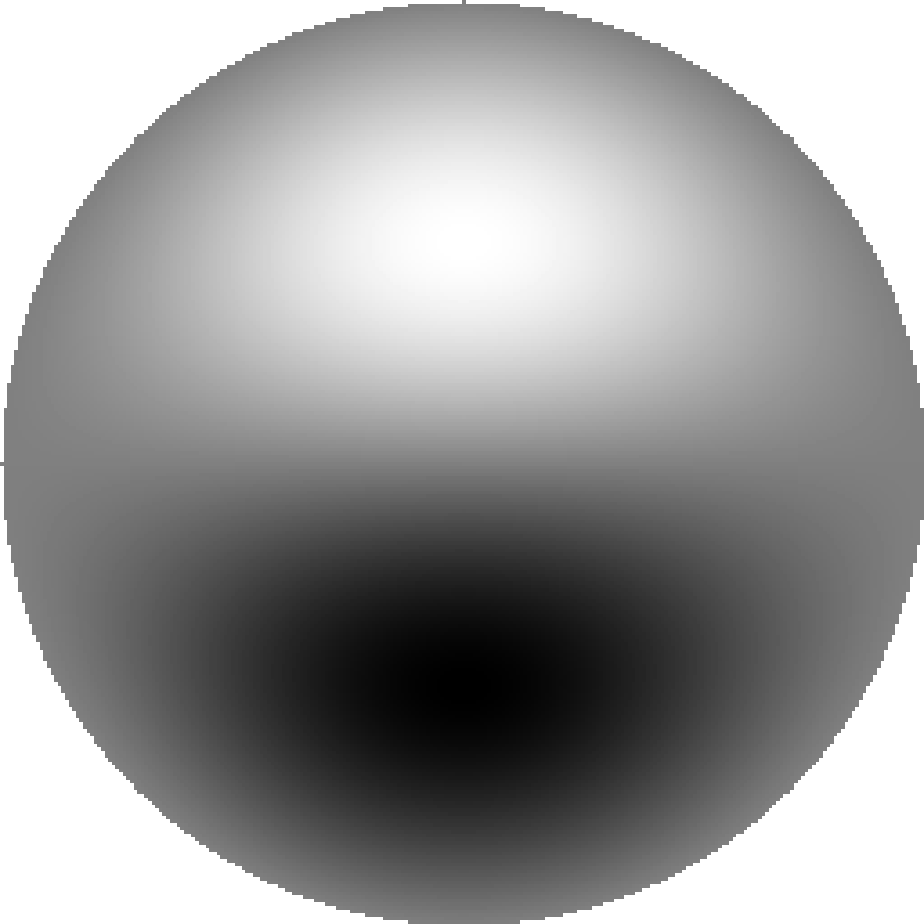}};
\node[anchor=north west] at (4.5,1.5) {\includegraphics[width=.15\textwidth]{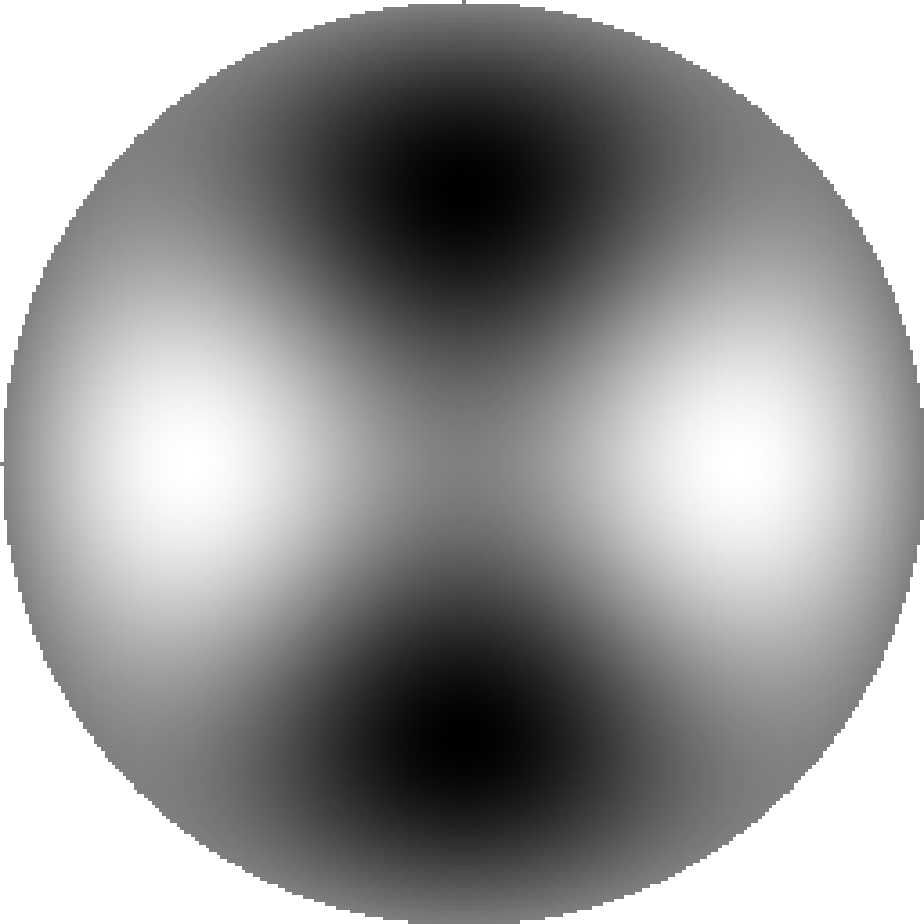}};
\node[anchor=north west] at (6,1.5) {\includegraphics[width=.15\textwidth]{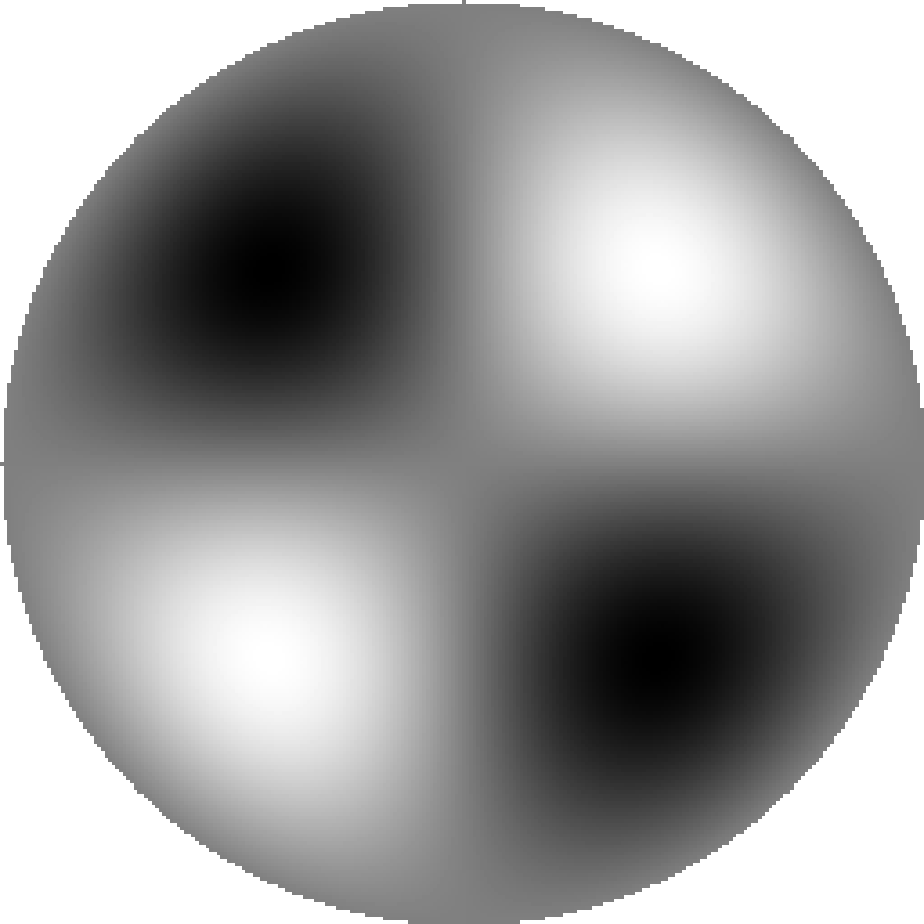}};
\node[anchor=north west] at (0,-0.3) {\includegraphics[width=.15\textwidth]{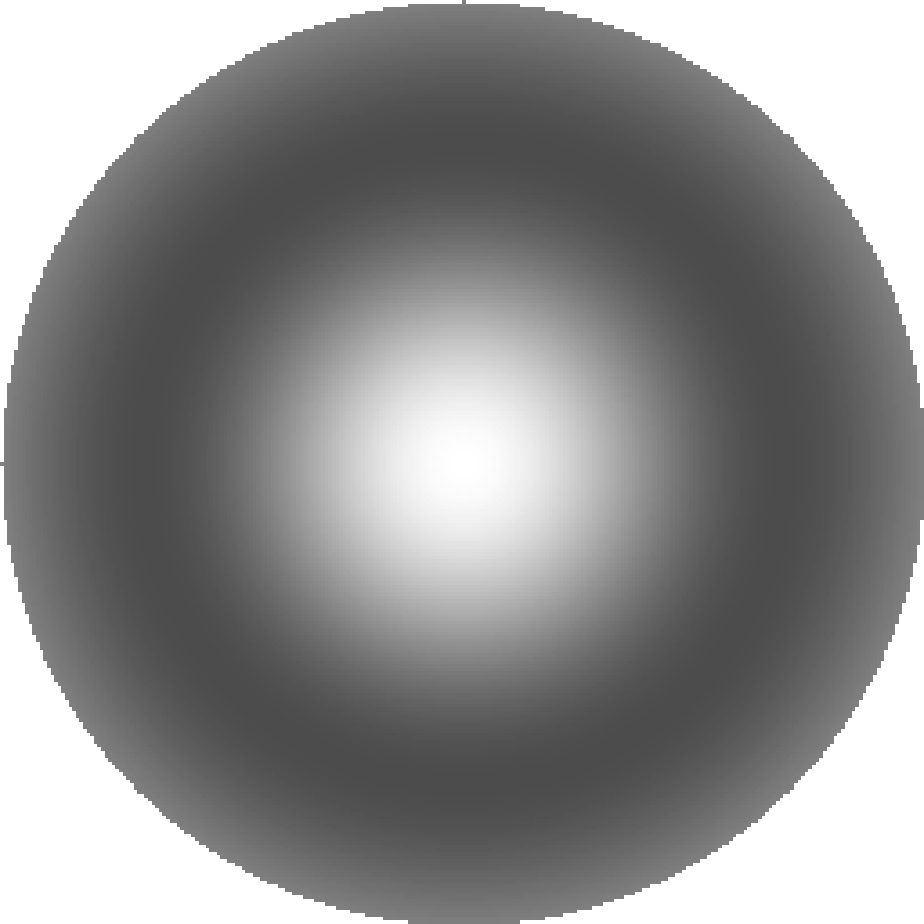}};
\node[anchor=north west] at (1.5,-0.3) {\includegraphics[width=.15\textwidth]{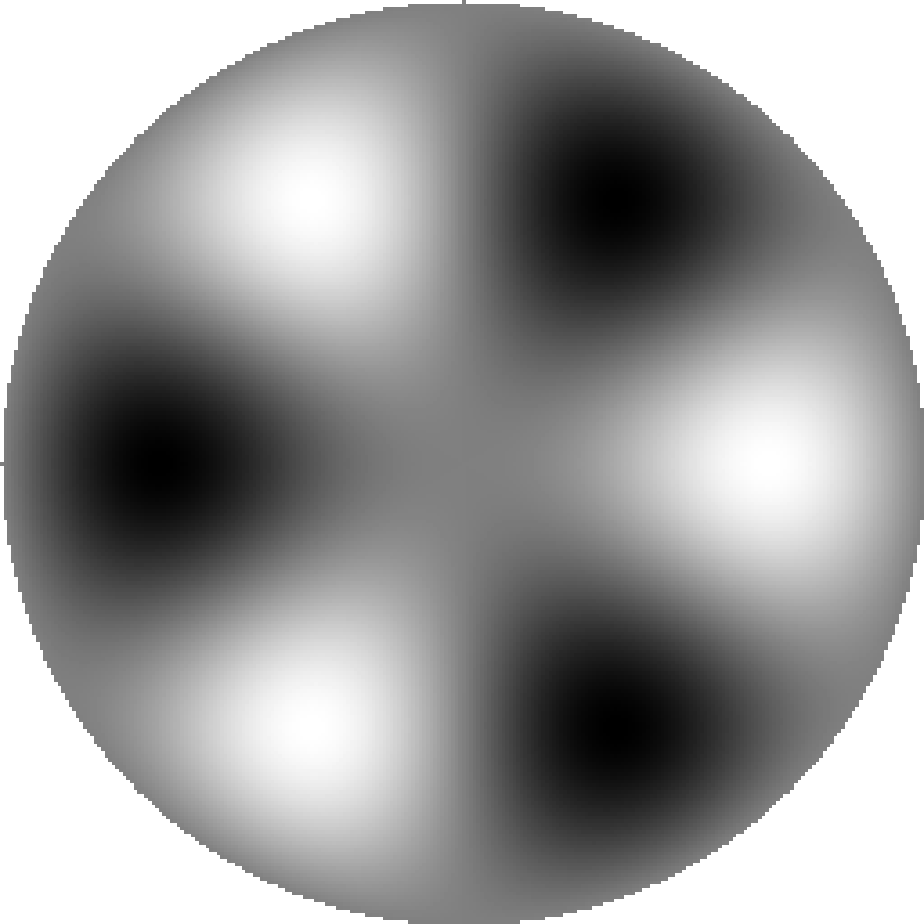}};
\node[anchor=north west] at (3,-0.3) {\includegraphics[width=.15\textwidth]{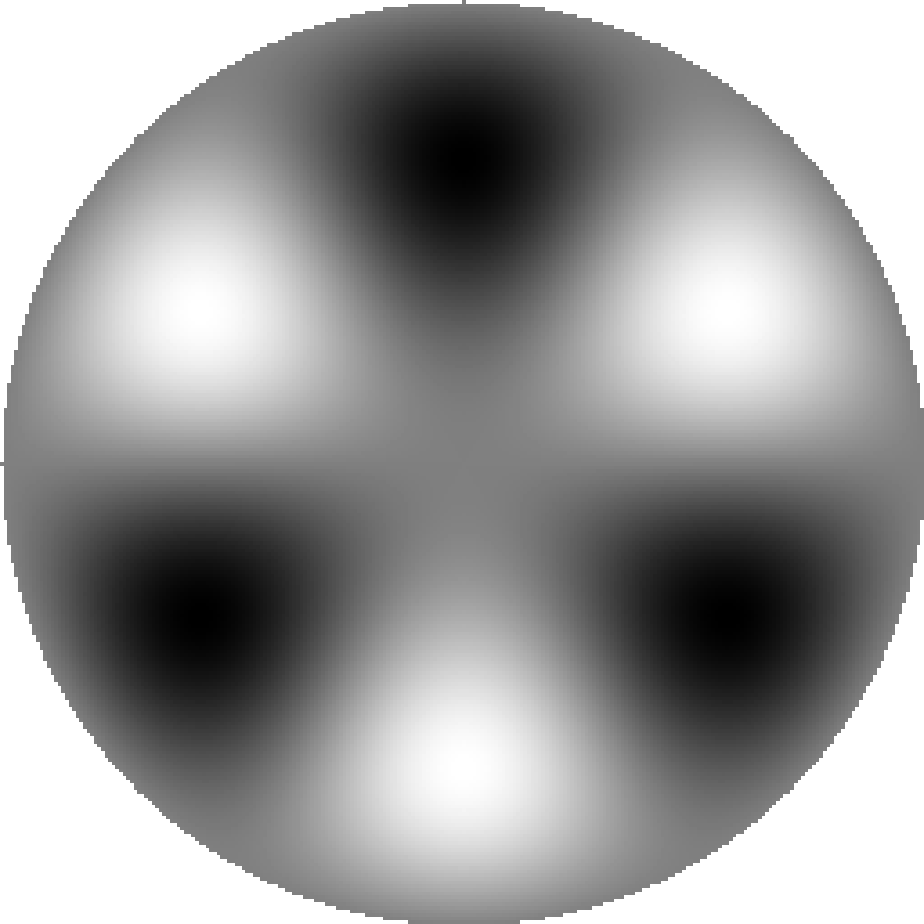}};
\node[anchor=north west] at (4.5,-0.3) {\includegraphics[width=.15\textwidth]{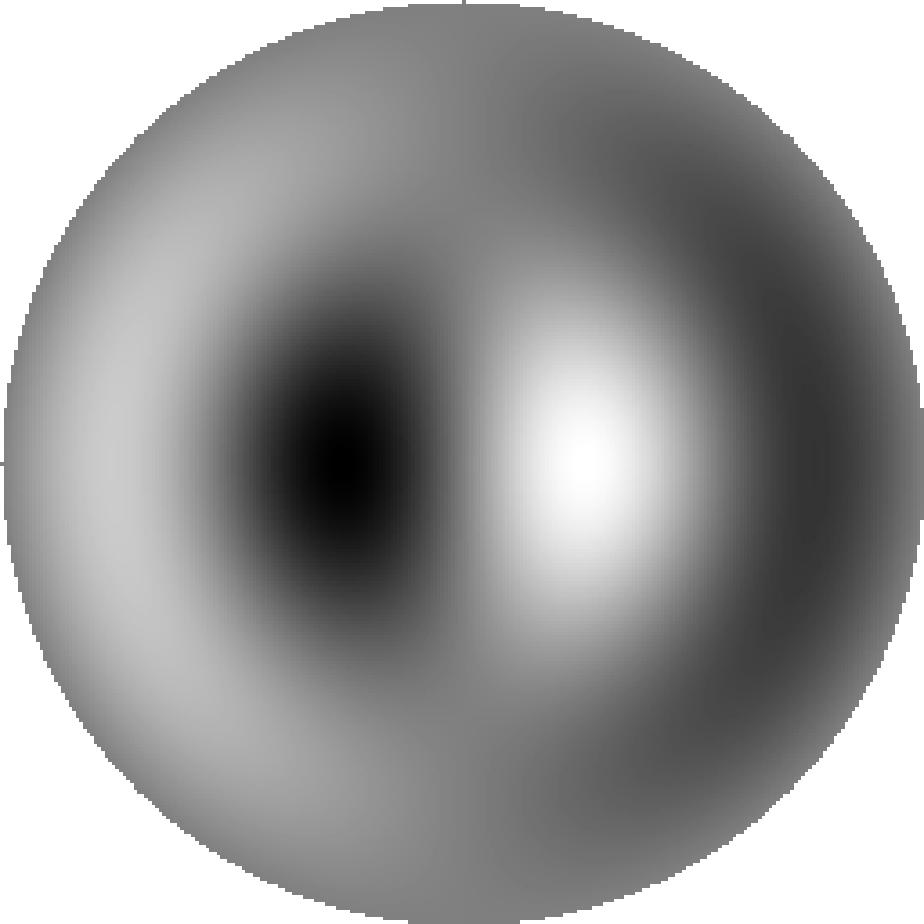}};
\node[anchor=north west] at (6,-0.3) {\includegraphics[width=.15\textwidth]{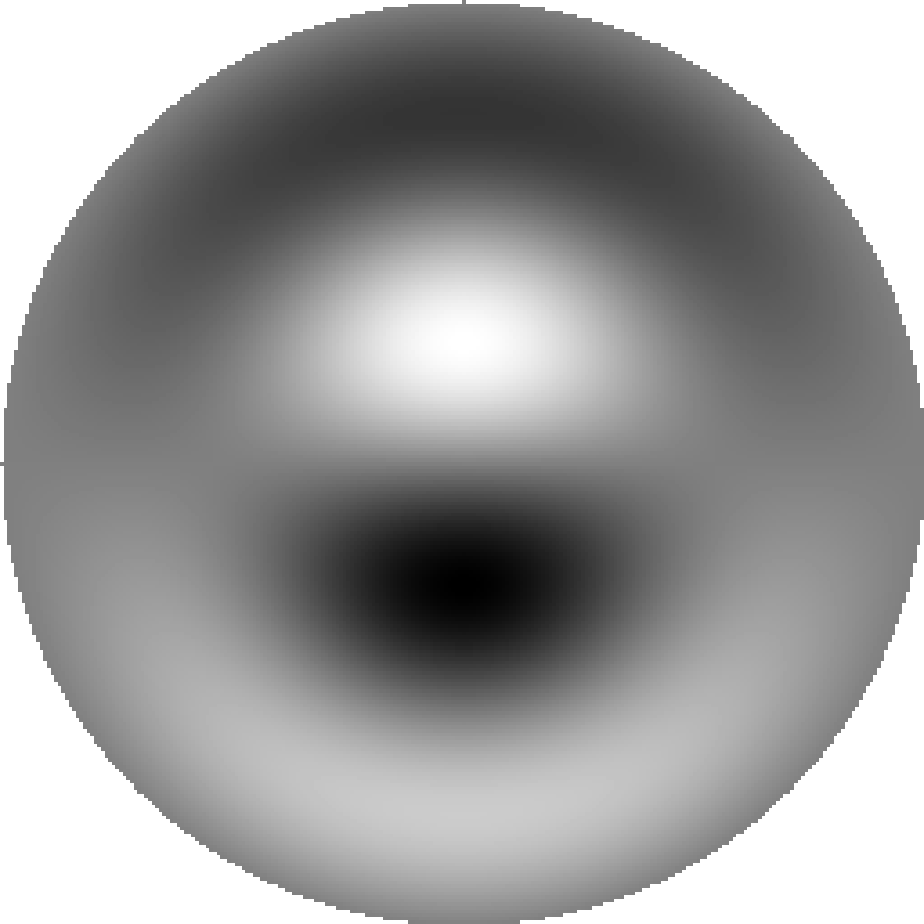}};
\node at (0.75,1.6) { $(0,1,2.4)$};
\node at (2.25,1.6) { $(1,1,3.8)$};
\node at (3.75,1.6) { $(-1,1,3.8)$};
\node at (5.25,1.6) { $(2,1,5.1)$};
\node at (6.75,1.6) { $(-2,1,5.1)$};
\node at (0.75,-0.2) { $(0,2,5.5)$};
\node at (2.25,-0.2) { $(3,1,6.4)$};
\node at (3.75,-0.2) { $(-3,1,6.4)$};
\node at (5.25,-0.2) { $(1,2,7.0)$};
\node at (6.75,-0.2) { $(-1,2,7.0)$};
\end{tikzpicture}
\caption{Illustration of the real version of eigenfunctions $\psi_{nk}$ (see Remark \ref{rmk:real_eigen}) associated with the smallest $10$ eigenvalues $\lambda_{nk}^2$. The triples above each figure show the corresponding values of $(n,k,\lambda_{nk})$, where $\lambda_{nk}$ is approximated to two digits of accuracy.}
 \label{fignew}
\end{figure}

\begin{theorem}[Informal Statement] \label{thm1informal}
Let $\varepsilon > 0$ be any fixed accuracy, and assume $m =\mathcal{O}( p)$. Then, algorithms described in \S \ref{sec:fast_details} apply the operators $B : \mathbb{C}^m \rightarrow \mathbb{C}^p$ and $B^* : \mathbb{C}^p \rightarrow \mathbb{C}^m$ with relative error less than $\varepsilon$ in $\mathcal{O}(p \log p)$ operations.
\end{theorem}

A more precise version of Theorem \ref{thm1informal} is stated in \S \ref{sec:proofmainresult} (see
Theorem \ref{thm1}), and supporting numerical results are reported in \S \ref{sec:numerics}. In particular, we show that the presented algorithm agrees with dense matrix multiplication for images with up to $p=160^2$ pixels and accuracy parameters as small as $\varepsilon = 10^{-14}$. We report timings of images with up to $p= 512^2$ (where constructing a dense transform matrix required a prohibitive amount of memory on our testing machine).  Moreover, we present a numerical example involving rotations, radial convolutions, and deconvolutions that illustrates the utility of this basis.

\begin{remark}[Real version of eigenfunctions] \label{rmk:real_eigen}
The complex eigenfunctions $\psi_{nk}$ can be transformed into real eigenfunctions $\tilde{\psi}_{nk}$ via the orthogonal transformation:
$$
\tilde{\psi}_{0k} = \psi_{0k}, \quad
\tilde{\psi}_{nk} = \frac{\psi_{nk} + (-1)^n \psi_{-nk}}{\sqrt{2}},
\quad \text{and} \quad
\tilde{\psi}_{-nk}
=
\frac{\psi_{nk} - (-1)^n \psi_{-nk}}{i \sqrt{2}},
$$
for $n \in \mathbb{Z}_{>0}$ and $k \in \mathbb{Z}_{>0}$; indeed, this follows from the definition \eqref{eq:eigenfun} of $\psi_{nk}$, the identify $J_{-n}(r) = (-1)^n J_n(r)$, and Euler's formula $e^{ix} = \cos x + i \sin x$.  
\end{remark}

\subsection{Relation to past work} \label{sec:past_work}
In this paper, we present a fast and accurate method to apply the operators $B$ and $B^*$
to vectors in
$\mathcal{O}(p \log p)$ operations for any fixed relative accuracy $\varepsilon$. We again emphasize that this is in contrast to previous results since, to the best of our knowledge, existing methods for computing the expansion coefficients in a steerable basis either require $\mathcal{O}(p^{3/2})$ operations \cite{landa2017approximation,landa2017steerable} or are heuristic in nature \cite{MR3472531,Zhao2014}. 
Further, we mention an interesting related work 
\cite{zhou2022spectral} appearing online after our work was posted on arXiv; it follows a similar approach to this paper but  considers a more general problem setting, where it achieves computational complexity $\mathcal{O}(p^{3/2})$.

The application of the operators $B$ and $B^*$ can be used in an iterative
method to determine least-squares optimal expansion coefficients for a given image, for instance, using modified Richardson iteration or the conjugate gradient method. Alternatively, 
applying $B^*$ to $f$ can be viewed as estimating the continuous inner products 
that define the coefficients by using quadrature points on a grid (and
potentially quadrature weights to provide an endpoint correction).

The most closely related previous approach \cite{MR3472531,Zhao2014} expands the Fourier-transform of images into the Fourier-Bessel basis by using a quadrature rule in the radial direction and an equispaced grid in the angular direction. This approach achieves complexity $\mathcal{O}(p\log p)$, but is heuristic and does not come with accuracy guarantees. Indeed, the authors of \cite{MR3472531,Zhao2014} do not claim any accuracy guarantees, and empirically the code associated with the paper has low numerical accuracy; part of the motivation for this paper is to make a fast and accurate method that is rigorously justified, and yields a code that agrees with direct calculation to close to machine precision.

\subsection{Organization} The remainder of the paper is organized as follows. 
In \S \ref{sec:analytic}, we describe the analytical apparatus underlying
the method. In \S \ref{sec:method}, we describe the
computational method. In \S \ref{sec:proofmainresult}, we state and prove Theorem \ref{thm1}, which is a more precise version of the informal result Theorem \ref{thm1informal}.
  In \S \ref{sec:numerics}, we present numerical results. In \S \ref{sec:discuss} we 
discuss the implications of the method and potential extensions.

\section{Analytical apparatus}\label{sec:analytic}
\subsection{Notation}
The  eigenfunctions of the Laplacian on the unit disk (that satisfy
Dirichlet boundary conditions) defined in \eqref{eq:eigenfun} can be extended to $\mathbb{R}^2$ as functions supported on the unit disk by
\begin{equation} \label{eq:eigenfun_extend}
\psi_{n k}(r,\theta) = c_{n k} J_n ( \lambda_{n k} r) e^{\imath n \theta} \chi_{[0,1)} (r), 
\end{equation}
for $(n,k) \in \mathbb{Z} \times \mathbb{Z}_{> 0}$, where $\chi_{[0,1)}$
denotes an indicator function for $[0,1)$. For the sake of
completeness, we note that the normalization constants $c_{nk}$
which ensure that $\|\psi_{nk}\|_{L^2} = 1$ are defined by
\begin{equation} \label{eq:eigenfun_const}
\begin{split}
        c_{nk} =
        \frac{1}{\pi^{1/2} |J_{n+1}(\lambda_{n k})|},
        \quad \text{for} \quad (n,k) \in \mathbb{Z} \times \mathbb{Z}_{>0},
        \end{split}
\end{equation}
see \cite[Eq. 10.6.3, Eq. 10.22.37]{dlmf}. 
We use the convention that
the Fourier 
transform $\widehat{f} : \mathbb{R}^2 \rightarrow \mathbb{C}$ of  an integrable
function $f : \mathbb{R}^2 \rightarrow \mathbb{C}$ is defined by
\begin{equation} \label{eq:fourier_transform}
\widehat{f}(\xi)  = \frac{1}{2\pi} \int_{\mathbb{R}^2} f(x) e^{-\imath x \cdot \xi} dx,
\end{equation}
where $x \cdot \xi$ denotes the Euclidean inner product. We define the convolution of two functions $f,g:\mathbb{R}^2 \rightarrow \mathbb{C}$ by
$$
(f * g)(x) = \int_{\mathbb{R}^2} f(x - y) g(y) dy.
$$
Furthermore, we will make use of the identity
\begin{equation}  \label{eq:jn_as_fourier_coeff}
J_n(r) = \frac{1}{2\pi} \int_0^{2\pi} e^{\imath r \sin \theta} e^{- \imath n \theta} d\theta,
\end{equation}
see for example \cite[Eq.~9.19]{temme1996special}. 
We note that the identities derived in the subsequent sections are similar to those derived in 
\cite{farashahi2021fourier,farashahi2019discrete,ghaani2018fourier,ghaani2022fourier}.

\subsection{Fourier transform of eigenfunctions}\label{sec:fourier_eigenfunc}
The analytic foundation for the presented fast method is the following expression 
for the Fourier transform of the functions $\psi_{nk}$ defined in  \eqref{eq:eigenfun_extend}, which we prove for completeness.

\begin{lemma}\label{lem:fourier_eigenfun}
The Fourier transform $\widehat{\psi}_{n k}$ can be expressed by
\begin{equation}\label{eq:fourier_eigenfun}
\widehat{\psi}_{n k}(\xi )
= (-\imath)^n e^{\imath n \phi } \int_0^1 c_{nk}J_n(\lambda_{n k} r) J_n( \rho r) r dr,
\end{equation}
where $(\rho,\phi)$ are polar coordinates for $\xi = (\rho \cos \phi, \rho \sin \phi)$.
\end{lemma}
\begin{proof}[Proof of Lemma \ref{lem:fourier_eigenfun}]
By the definition of the Fourier transform \eqref{eq:fourier_transform} we have
$$
\widehat{\psi}_{nk}(\xi) = \frac{1}{2\pi} \int_{\mathbb{R}^2} \psi_{nk}(x) e^{-\imath x \cdot \xi} dx.
$$
Changing to polar coordinates $\xi = (\rho \cos \phi, \rho \sin \phi)$ and $x = (r \cos \theta, r \sin \theta)$ gives
$$
\widehat{\psi}_{n k}(\xi) = \frac{1}{2\pi} 
\int_0^{2\pi}  \int_0^1 c_{n k} J_n(\lambda_{n k} r) e^{\imath n \theta} e^{-\imath r \rho \cos(\theta - \phi)} r dr d\theta ,
$$
where we used the fact that $x \cdot \xi = r \rho \cos(\theta - \phi)$. Changing variables $\theta \mapsto -\theta + \phi - \pi/2$ and taking the integral over $\theta$ gives
$$
\widehat{\psi}_{n k}(\xi) 
= (-\imath)^n e^{\imath n \phi }  \int_0^1 c_{nk}J_n(\lambda_{n k} r) J_n( \rho r) r dr,
$$
as desired.
\end{proof}

\subsection{Coefficients from eigenfunction Fourier transform}
Next, we observe how the coefficients of a function in the
eigenfunction basis can be computed by an application of
Lemma~\ref{lem:fourier_eigenfun}. In the following, we will write the arguments of Fourier transforms of functions in polar coordinates $(\rho, \phi)$. 
We have the following result:

\begin{lemma}\label{lem:coeff_from_beta}
Suppose that $\mathcal{I} \subset \mathbb{Z} \times \mathbb{Z}_{>0}$ is a finite index set, and set
\begin{equation} \label{eq:f_in_basis}
f = \sum_{(n,k) \in \mathcal{I}} \alpha_{nk} \psi_{n k},
\end{equation}
where $\alpha_{nk} \in \mathbb{C}$ are coefficients. Define $\beta_n : [0,\infty) \rightarrow \mathbb{C}$ by
\begin{equation}\label{eq:def_beta}
\beta_n(\rho) := \imath^n \int_0^{2\pi} \widehat{f}(\rho,\phi) e^{-\imath n \phi} d\phi.
\end{equation}
It then holds that
\begin{equation}\label{eq:beta_from_alpha}
\alpha_{n k} = c_{n k}\beta_n(\lambda_{n k}).
\end{equation}
\end{lemma}

The proof is a direct consequence of Lemma \ref{lem:fourier_eigenfun}.

\begin{proof}[Proof of Lemma \ref{lem:coeff_from_beta}]
Observe that \eqref{eq:fourier_eigenfun} implies
\begin{equation}\label{eq:first_id}
    \imath^n \int_0^{2\pi} \widehat{\psi}_{n' k'}(\rho,\phi)e^{-\imath n \phi} d\phi = 2\pi \delta_{n,n'} \int_0^1 c_{n' k'} J_{n'}(\lambda_{n' k'} r) J_{n'}(\rho r) r dr,
\end{equation}
where $\delta_{n,n'} =1$ if $n = n'$ and $\delta_{n,n'} = 0$ otherwise.
Evaluating \eqref{eq:first_id} at radius $\rho = \lambda_{n k}$ gives
\begin{equation*}
\begin{split}
& \imath^n \int_0^{2\pi} \widehat{\psi}_{n' k'}(\lambda_{n k},\phi) e^{-\imath n \phi} d\phi  = 2\pi \delta_{n,n'}\int_0^1 c_{n' k'} J_{n'}(\lambda_{n' k'} r) J_{n'}(\lambda_{nk} r) r dr \\
&= 2\pi \delta_{n,n'}\int_0^1 c_{n k'} J_n(\lambda_{n k'} r) J_n(\lambda_{nk} r) r dr =  \frac{1}{c_{n k}} \delta_{n,n'} \delta_{k,k'},
\end{split}
\end{equation*}
where the final equality follows from the orthogonality of the eigenfunctions $\psi_{nk'}$ (which is a consequence of the fact that the Laplacian is self-adjoint). By the definition of $\beta_n$ in \eqref{eq:def_beta}, this implies that
\begin{equation*}
\begin{split}
    \beta_n(\lambda_{n k}) &= \!\!\!\! \sum_{(n',k') \in \mathcal{I}} \!\!\!\! \alpha_{n'k'} \imath^{n} \int_0^{2\pi} \!\! \widehat{\psi}_{n' k'}(\lambda_{nk},\phi)e^{-\imath n \phi} d\phi = \!\!\!\! \sum_{(n',k') \in \mathcal{I}} \frac{\alpha_{n'k'}}{c_{n'k'}}\delta_{n,n'} \delta_{k,k'} = \frac{\alpha_{nk}}{c_{nk}},
    \end{split}
\end{equation*}
which concludes the proof.
\end{proof}

\begin{remark}[Special property of Bessel functions]
We emphasize that the integral expression \eqref{eq:jn_as_fourier_coeff} of the Bessel function is crucial for the fast method of this paper. The possibility of extending the approach to create other fast transforms defined on domains in $\mathbb{R}^2$, therefore hinges on identifying equally useful integral expressions for the corresponding transforms.
\end{remark}

\subsection{Convolution with radial functions}
Let $g(x) = g(|x|)$ be a radial function. In this section, we observe how the convolution 
with $g$ can be computed via a diagonal transform of the coefficients.
More precisely, we compute the projection of the convolution with $g$ onto the span of any finite basis of the eigenfunctions $\psi_{nk}$. 
\begin{lemma} \label{lem:radial_conv}
Let $f$ be a function with coefficients $\alpha_{n k}$ as in \eqref{eq:f_in_basis}, and $g(x) = g(|x|)$ be a
radial function. We have
$$
P_\mathcal{I} (f*g) = \sum_{(n,k) \in \mathcal{I}} \alpha_{nk} \widehat{g}(\lambda_{nk}) \psi_{n k},
$$
where $P_{\mathcal{I}}$ denotes the 
orthogonal projection onto the span of 
$\{\psi_{n k}\}_{(n,k) \in \mathcal{I}}$. 
\end{lemma}
The proof is a direct application of Lemma \ref{lem:coeff_from_beta}.
\begin{proof}[Proof of Lemma \ref{lem:radial_conv}]
We use the notation $g(x) = g(|x|)$ and $\widehat{g}(\xi) = \widehat{g}(|\xi|)$.
Since the functions $\psi_{n k}$ are an orthonormal basis, in order to compute
the orthogonal projection $P_{\mathcal{I}}$, it  suffices to determine the
coefficients of $f * g$ with respect to $\psi_{n k}$ for $(n,k) \in \mathcal{I}$.
Since
$\widehat{(f * g)}(\rho, \phi) = \widehat{f}(\rho, \phi)\widehat{g}(\rho)
$, and $\widehat{g}$ is radial, we have 
\begin{equation*}
\begin{split}
    \imath^n \int_{0}^{2\pi}\widehat{(f * g)}(\lambda_{nk}, \phi) e^{-\imath n\phi} d\phi &= \imath^n \int_{0}^{2\pi}\widehat{f}(\lambda_{nk}, \phi)  \widehat{g}(\lambda_{nk}) e^{-\imath n\phi} d\phi = \frac{\alpha_{nk}}{c_{nk}} \widehat{g}(\lambda_{nk}),
    \end{split}
\end{equation*}
where the final equality follows from \eqref{eq:beta_from_alpha}. An application of Lemma~\ref{lem:coeff_from_beta} then completes the proof.
\end{proof}
\subsection{Maximum bandlimit} \label{sec:bandlimit}
In this section, we use Weyl's law and lattice point counting estimates to derive a bound on the bandlimit parameter $\lambda$ in terms of the number of pixels $p$ under the assumption that the number of basis functions should not exceed the number of pixels corresponding to points in the unit disk. 

Recall from \eqref{eq:m} that the number of basis functions $m$ is determined from $\lambda$ by
$$
m = \# \{ (n,k) \in \mathbb{Z} \times \mathbb{Z}_{>0} : \lambda_{n k} \le \lambda \},
$$
where $\lambda_{nk}$ is the $k$-th smallest positive root of $J_n$.
Further, recall that  $\lambda_{nk}^2$ are the eigenvalues of the Dirichlet Laplacian on the unit disk, see \eqref{eq:eigenval}. Thus, it follows from Weyl's law that
\begin{equation}\label{eq:weyl}
    \# \{ (n,k) \in \mathbb{Z}\times \mathbb{Z}_{>0} : \lambda_{nk} \leq \lambda \} = \frac{\lambda^2}{4} - \frac{\lambda}{2} + \mathcal{O}(\lambda^{2/3}),
\end{equation}
see \cite{colin2010remainder}. On the other hand, the number of pixels representing points in the unit disk is equal to the number of integer lattice points from $\mathbb{Z}^2$ inside a disk of radius $\lfloor (\sqrt{p}+1)/2 \rfloor$, see \eqref{eq:pixel_locs}. Classic lattice point counting results give
\begin{equation}\label{eq:num_points_circle}
    \# \left\{ (j_1,j_2) \in \mathbb{Z}\times \mathbb{Z} : j_1^2 + j_2^2 \leq  \left\lfloor \frac{\sqrt{p}+1}{2} \right\rfloor ^2 \right\} = \pi \left\lfloor \frac{\sqrt{p}+1}{2} \right\rfloor ^2 + \mathcal{O}(p^{1/3}),
\end{equation}
see for example \cite{gauss_2011,hardy1924lattice,van1923neue}. Equating~\eqref{eq:weyl} with \eqref{eq:num_points_circle} results in
\begin{equation}\label{eq:bandlimit}
    \lambda = 2\sqrt{\pi} \left\lfloor \frac{\sqrt{p}+1}{2} \right\rfloor + 1 + \mathcal{O}\left(p^{-1/6}\right).
\end{equation}
For simplicity, motivated by \eqref{eq:bandlimit}, we assume 
\begin{equation} \label{eq:bandlimit_simple}
\lambda \le \sqrt{\pi p}.
\end{equation}
Practically speaking, it can be advantageous to expand the image using fewer basis functions than described by \eqref{eq:bandlimit}. See, for example, the heuristic described by Remark \ref{rmk:bandlimit_heuristic} or see \cite{MR3472531}. 

\section{Computational method} \label{sec:method}
In this section, we describe how to apply the
operators $B$ and $B^*$ defined above in \S \ref{sec:main_result} in $\mathcal{O}(p \log p)$ 
operations. For the purpose of exposition,
we start by describing a simplified method before presenting the full method. The
section is organized as follows:
\begin{itemize}
\item In \S \ref{sec:algo_notation}, we introduce notation for the algorithm description.
\item  In \S \ref{sec:simplified}, we give an informal description of a simplified method to apply
$B$ and $B^*$ in $\mathcal{O}(p^{3/2} \log p)$ operations. The simplified method is a
direct application of the lemmas from the previous section.
\item In \S \ref{sec:fast_summary}, we provide an informal description of how to
modify the simplified method to create a fast method to apply $B$ and $B^*$ in
$\mathcal{O}(p \log p)$ operations. The main additional ingredient is a fast method of
interpolation from Chebyshev nodes.
\item In \S \ref{sec:fast_details}, we give a detailed description of the fast 
method to apply $B$ and $B^*$ in $\mathcal{O}(p \log p)$ operations.
\end{itemize}
%
\subsection{Notation} \label{sec:algo_notation}
Recall that $x_1,\ldots,x_p$ and $f_1,\ldots,f_p$ are an enumeration of the pixel locations and corresponding pixel values, and $\lambda_1,\ldots,\lambda_m$ and $\psi_1,\ldots,\psi_m$ are an enumeration of the Bessel function roots and corresponding eigenfunctions, see \S \ref{sec:notation}. Let
$c_1,\ldots,c_m$
be an enumeration of the normalization constants defined in \eqref{eq:eigenfun_const} such that $c_j$ is the normalization constant associated with $\psi_j$, and let
$n_1,\ldots,n_m$ and $k_1,\ldots,k_m,$
be an enumeration of the Bessel function orders and root numbers such that $\psi_{n_j k_j} = \psi_j$.
Further, we define
$N_m = \max \{ n_j \in \mathbb{Z} : j \in \{1,\ldots,p\}\}$ to be the maximum order of the Bessel functions,  and 
$K_n := \max \{ k \in \mathbb{Z}_{>0} :\lambda_{nk} \le \lambda \}$ for $n \in \{-N_m, \ldots , N_m\}$.

A key ingredient in the simplified and fast methods is the  non-uniform fast Fourier transform (NUFFT) \cite{Dutt1993,Greengard2004,lee2005type}, which is a now standard tool in computational mathematics. Given $n$ source points and $m$ target points in $\mathbb{R}^d$, and $1 \ge \varepsilon > 0$, the NUFFT involves
\begin{equation} \label{eq:nufft}
\textstyle
\mathcal{O}\left(n \log n + m \left( \log \frac{1}{\varepsilon}\right)^d\right)
\end{equation}
operations to achieve $\ell^1$-$\ell^\infty$ relative error $\varepsilon$, see \cite[Eq. (9)]{barnett2021aliasing}.  Throughout the paper (except for
Theorem \ref{thm1} and its proof), we treat $\varepsilon$ as a fixed constant, say, $\varepsilon=10^{-7}$, and do not include it in computational complexity statements. 

\subsection{Informal description of simplified method} \label{sec:simplified}
In this section,  we present a simplified  
method that applies $B$ and $B^*$ in $\mathcal{O}(p^{3/2}
\log p )$ operations.  We first describe how to apply $B^*$.
The basic idea is to apply Lemma~\ref{lem:coeff_from_beta} to the function
$$
f(x) = \sum_{i=1}^p f_j \delta(x - x_j),
$$
where $\delta$ is a Dirac delta distribution. Observe that, by our convention \eqref{eq:fourier_transform}, the Fourier transform of $f$ is  
$$
\widehat{f}(\xi) = \frac{1}{2\pi} \sum_{j=1}^p  f_j e^{-i x_j \cdot \xi}.
$$
In polar coordinates $x_j = (r_j \cos \theta_j, r_j \sin \theta_j)$ and $\xi = (\rho \cos \phi, \rho \sin \phi)$ 
$$
\widehat{f}(\rho, \phi) = 
\frac{1}{2\pi} \sum_{j=1}^p  f_j e^{-i r_j \rho \cos(\theta_j - \phi)},
$$
and by the definition \eqref{eq:def_beta} of $\beta_n$ we have
\begin{equation} \label{eq:int_over_phi}
\beta_n(\rho) = \sum_{j=1}^p f_j \frac{i^n}{2\pi} \int_0^{2\pi} e^{-\imath r_j \rho \cos(\theta_j - \phi)} e^{-i n \phi} d\phi.
\end{equation}
Changing variables $\phi \mapsto \phi + \theta_j +\pi/2$ and using the identity \eqref{eq:jn_as_fourier_coeff} gives
$$
\beta_n(\rho) = \sum_{j=1}^p f_j J_n(r_j \rho) e^{-i n \theta_j}.
$$
By the definition \eqref{eq:operator_B^*} of $B^*$ it follows that
$$
(B^* f)_i = c_i \beta_{n_i}(\lambda_i) h.
$$
In order to implement the above calculations numerically, we need to discretize the integral in \eqref{eq:int_over_phi}.
In Lemma \ref{lem:num_angular_nodes}, we prove that discretizing $\phi$ using $s = \mathcal{O}(\sqrt{p})$ equispaced angles guarantees that sums over the equispaced angles approximate integrals over $\phi$ to sufficient accuracy.
In more detail, the simplified method for applying $B^*$ can be described as follows: 
\begin{enumerate}[label={\textbf{Step \arabic*.}}, parsep=2ex,wide=0pt]
\item Using the type-2 2-D NUFFT compute:
$$
a_{i \ell} := \sum_{j=1}^p f_j e^{-\imath x_j \cdot \xi_{i \ell}}
\quad \text{where} \quad
\xi_{i \ell} := \lambda_i (\cos \phi_\ell, \sin \phi_\ell),
$$
for $(i,\ell) \in \{1,\ldots,m\} \times \{0,\ldots,s-1\}$, where $\phi_\ell = 2\pi \ell/s$.
The computational complexity of this step is $\mathcal{O}(p^{3/2})$ using the NUFFT
since there are $\mathcal{O}(p)$ source nodes, and $\mathcal{O}(p^{3/2})$ target nodes, see \eqref{eq:nufft}.
\item  Using the FFT compute:
$$
\beta_{n}(\lambda_{i}) \approx \frac{\imath^n}{ s}\sum_{\ell=0}^{s-1} a_{i \ell} e^{-\imath n \phi_\ell }
$$
for $(i,n) \in \{1,\ldots,m\} \times \{0,\ldots,{s}-1\}$. 
Since this step involves $m = \mathcal{O}(p)$ FFTs of size $s = \mathcal{O}(\sqrt{p})$, the computational complexity of this step is $\mathcal{O}(p^{3/2} \log p)$. 

\item By Lemma \ref{lem:coeff_from_beta}, it follows that
$$
(B^* f)_i = \beta_{n_i}(\lambda_i) c_i h,
$$
for $i \in \{1,\ldots,m\}$. The computational complexity of this step is $\mathcal{O}(p)$ since it only involves selecting and scaling $\beta_{n_i}(\lambda_i)$.
\end{enumerate}
\subsection{Sketch of fast method} \label{sec:fast_summary}
In this section, we describe how the computational complexity of the simplified method of the previous section for applying $B^*$ can be
improved from $\mathcal{O}(p^{3/2} \log p)$ to $\mathcal{O}(p \log p)$ by using
fast interpolation from Chebyshev nodes.  

The problem with the simplified method is the first step: it involves computing 
$\widehat{f}(\rho,\phi)$ for 
$\mathcal{O}(p)$ values of $\rho$ and $\mathcal{O}(\sqrt{p})$ values of $\phi$ for a total of $\mathcal{O}(p^{3/2})$ points, which is already prohibitively expensive.
Fortunately, there is a simple potential solution to this problem: since the functions $\beta_n(\rho)$ are analytic functions of $\rho$, it might be possible to tabulate them at appropriate points and then use polynomial interpolation to compute the coefficients. We take this approach to design a fast method. Crucially, we prove that tabulating each $\beta_n$ at $\mathcal{O}(\sqrt{p})$ Chebyshev nodes is sufficient to achieve the desired accuracy (see Lemma \ref{lem:num_radial_nodes}).  This reduces the number of target points in the NUFFT in the first step of the algorithm from $\mathcal{O}(p^{3/2})$ to $\mathcal{O}(p)$. Note that one should not expect to be able to use $o(p)$ points in total, since the images have $p$ pixels.

In more detail, here is an informal summary of
the fast method:
\begin{itemize}
\item Compute the Fourier transform of $f$ at 
$$
\xi_{ k \ell} := t_k (\cos \phi_\ell, \sin \phi_\ell),
$$
for $\mathcal{O}(\sqrt{p})$ Chebyshev nodes $t_k$ and $\mathcal{O}(\sqrt{p})$ angles $\phi_\ell$.
\item Approximate $\beta_n(t_k)$
for the $\mathcal{O}(\sqrt{p})$ Chebyshev nodes $t_k$ for the interval $[\lambda_1,\lambda_m]$ and $\mathcal{O}(\sqrt{p})$ frequencies $n$.
\item For each of the $\mathcal{O}(\sqrt{p})$ frequencies, use fast
interpolation from the $\mathcal{O}(\sqrt{p})$ Chebyshev nodes $t_k$ to the
$\mathcal{O}(\sqrt{p})$ Bessel function roots associated with each frequency $n$.
We illustrate the interpolation step in Fig. \ref{fig:02}.
\end{itemize}
\begin{figure}[h!]
\centering
\begin{tabular}{c}
\includegraphics[width=.7\textwidth]{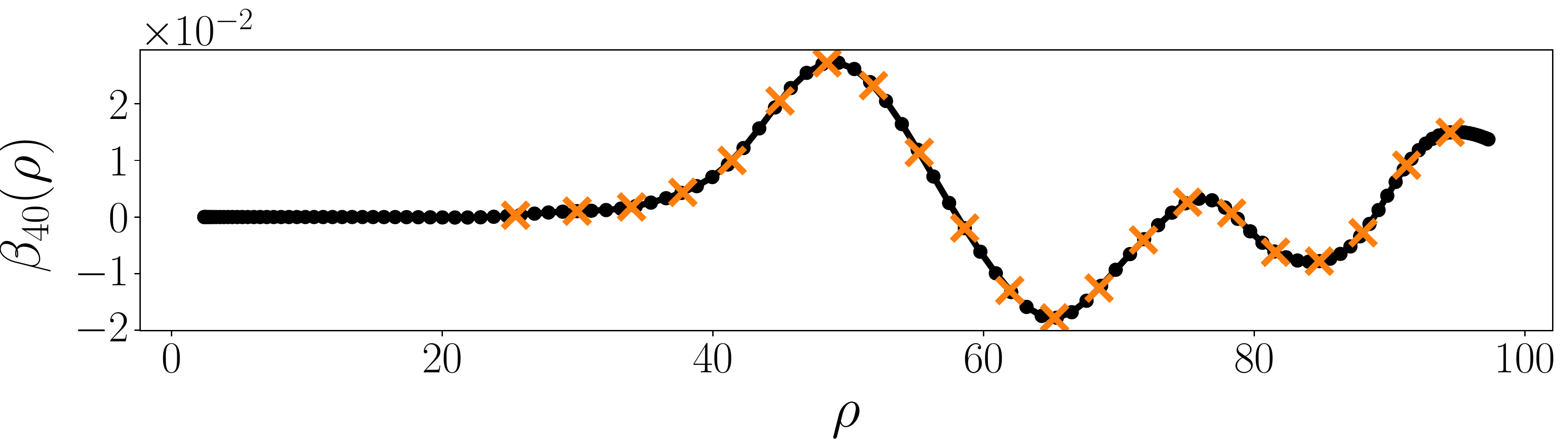} \\
\includegraphics[width=.7\textwidth]{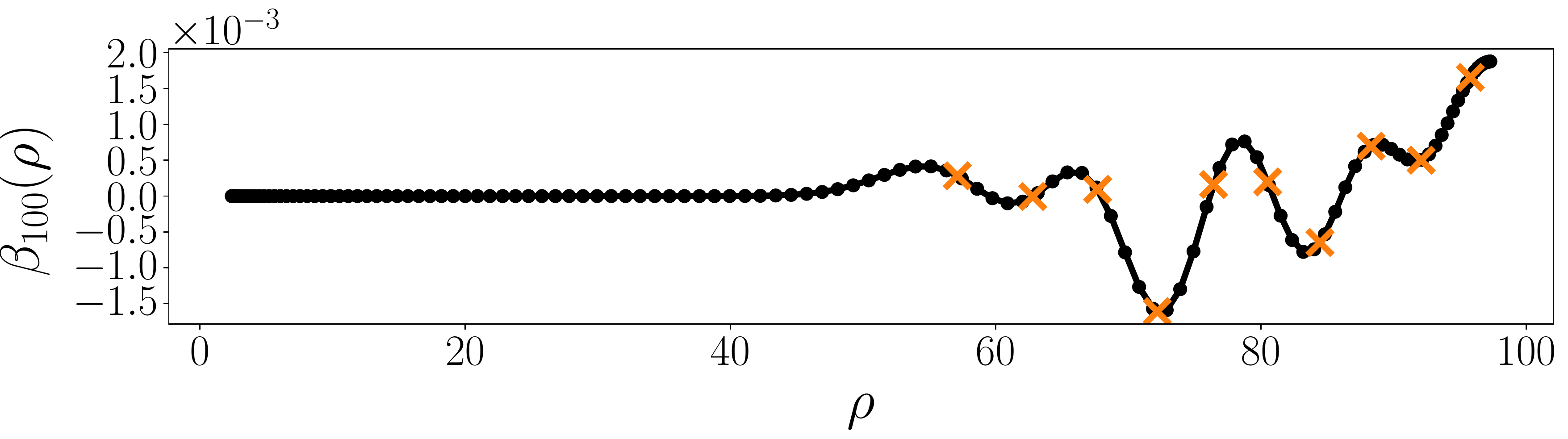}
\end{tabular}
\caption{We visualize the interpolation
step for a $64 \times 64$ input image. 
For $n=40$ and $100$  we plot
$\beta_n(\rho)$ (black line), interpolation source nodes (black dots), and target points (orange crosses).}  \label{fig:02}
\end{figure}

\subsection{Detailed description of fast method} \label{sec:fast_details}
In addition to the notation of
\S
\ref{sec:algo_notation}, let
\begin{equation}\label{eq:def_chebyshev}
    t_k := \frac{\lambda_m - \lambda_1}{2}\cos\left(\frac{2k+1}{q}\cdot \frac{\pi}{2}\right) + \frac{\lambda_1 + \lambda_m}{2}, \quad k = 0, \ldots , q-1,
\end{equation}
be Chebyshev nodes of the first kind in the interval $[\lambda_1, \lambda_m]$, for fixed integer $q$. We present a detailed
 description of the fast method for applying $B^*$ in Algorithm \ref{algo}.

\begin{algorithm}[h!] \label{algo} 
\caption{Fast method for applying $B^*$.}
\KwIn{Image $f$, bandlimit $\lambda$, and accuracy parameter $\varepsilon$.}
\Constants{$\#$ of pixels $p$, $\#$ of basis functions $m$, 
$\varepsilon^\text{dis}$ defined by \eqref{eq:epsdis},
$\varepsilon^\text{nuf}$ and $\varepsilon^\text{fst}$ defined by \eqref{eq:nufst},
$
\textstyle
s = \lceil \max\{ 7.09 \sqrt{p} , |\log_2 \varepsilon^\text{dis}| \} \rceil,\,$ and
$q = \lceil \max\left\{ 2.4 \sqrt{p} , |\log_2 \varepsilon^\text{dis}| \right\} \rceil.
$
}
\KwOut{$\alpha$ approximating $B^* f$ to relative error $\varepsilon$ (see Theorem \ref{thm1}).}
Using NUFFT, compute:
$$
a_{k \ell} := \sum_{j=1}^p f_j e^{-\imath x_j \cdot \xi_{k \ell}},
$$
with relative error $\varepsilon^\text{nuf}$, where
$\xi_{k \ell} := t_k (\cos \phi_\ell, \sin \phi_\ell)$ and $\phi_\ell = 2\pi \ell/s$,
for $k \in \{0,\ldots,q-1\}$ and $\ell \in \{0,\ldots,s-1\}$.
\\
Using FFT, compute: 
$$
\beta_{n k} := \frac{ \imath^n}{ s}\sum_{\ell=0}^{s-1} a_{k \ell} e^{-\imath n \phi_\ell},
$$
for $k \in \{0,\ldots,q-1\}$ and $n \in \{-N_m,\ldots,N_m\}$.
\\
 Using fast Chebyshev interpolation, compute:
$$
\alpha_i := 
\sum_{k=0}^{q-1} c_i \beta_{n_i k} u_k(\lambda_i)  h,
$$
for $i \in \{1,\ldots,m\}$ with relative error $\varepsilon^\text{fst}$ for $k \in \{0,\ldots,q-1\}$ and $n \in \{-N_m,\ldots,N_m\}$,  where $u_k(t)$ is defined in \eqref{eq:def_inter_poly}.
\\[5pt]
\end{algorithm}
 
 \begin{remark}[Methods for fast interpolation from Chebyshev nodes] \label{rmk:fast_interp}
Given values $v_0,\ldots,v_{q-1}$, we denote by $P$ the $q-1$ degree polynomial such that $P(t_k) = v_k$ for $k \in \{0,\ldots,q-1\}$, where $t_k$ are the Chebyshev nodes defined in
\eqref{eq:def_chebyshev}. We can explicitly write $P$ as
\begin{equation}\label{eq:def_inter_poly}
P(t) = \sum_{k=0}^{q-1} v_k u_k(t),
\quad \text{where} \quad
u_k(t) = \frac{\prod_{\ell \neq k} (t-t_\ell)}{\prod_{\ell \neq k} (t_k-t_\ell)}, \,\, \text{for $k \in \{0,\ldots,q-1\}$}.
\end{equation}
 Given $r$ target points $w_0,\ldots,w_{r-1}$, the map $(v_0,\ldots,v_{q-1})\mapsto(P(w_0),\ldots,P(w_{r-1}))$ is a linear mapping $\mathbb{C}^q \rightarrow \mathbb{C}^{r}$.
This linear operator (and its adjoint) can be applied fast by a variety of methods: for example, the interpolation could be performed by using the NUFFT \cite{Dutt1993,Greengard2004,lee2005type} in $\mathcal{O}(p \log p)$ operations (which is often called spectral interpolation), the Fast Multipole Method (FMM) \cite{dutt1996} in $\mathcal{O}(p)$ operations, or generalized Gaussian quadrature \cite{gimbutas2020fast} in $\mathcal{O}( p \log p )$ operations. Although the FMM has the lowest computational complexity, it is known to have a large run-time constant, so using other methods may be faster in applications. Practically speaking,  
choosing a fixed number of
source points centered around each target point (say $20$ source
points) and then applying a precomputed sparse (barycentric interpolation \cite{berrut2004barycentric})
matrix may be more practical than any of these methods; sparse interpolation can be used in combination with spectral interpolation  (by discrete cosine transform) to first increase the number of Chebyshev nodes.
\end{remark}
\begin{algorithm}[h!] \label{algoB} 
\caption{Fast method for applying $B$.}
\KwIn{Coefficients $\alpha$, bandlimit $\lambda$, and accuracy parameter $\varepsilon$.}
\Constants{$\#$ of pixels $p$, $\#$ of basis functions $m$, $\varepsilon^\text{dis}$ defined by \eqref{eq:epsdis2},
$\varepsilon^\text{fst}$ and
$\varepsilon^\text{nuf}$ defined by \eqref{eq:error_nufft_B_proof},
$
\textstyle
s = \lceil \max\{ 7.09 \sqrt{p} , |\log_2 \varepsilon^\text{dis}| \} \rceil,\,$ and $q = \lceil \max\left\{ 2.4 \sqrt{p} , |\log_2 \varepsilon^\text{dis}| \right\} \rceil.
$}

\KwOut{$f$ approximating $B \alpha$ to  relative error $\varepsilon$ (see Theorem \ref{thm1})}
Using a fast Chebyshev interpolation method, compute
$$
\beta^*_{nk} = h \sum_{i : n_i = n}  u_k(\lambda_i) c_i \alpha_i,
$$
with relative error less than $\varepsilon^\text{fst}$ 
for $n \in \{-N_m,\ldots,N_m\}$ and $k \in \{0,\ldots,q-1\}$, where $u_k(t)$ is defined in \eqref{eq:def_inter_poly}.

Using FFT compute
$$
a^*_{k \ell} := \sum_{n=-N_m}^{N_m} \frac{ (-\imath)^n}{ s}
e^{\imath n \phi_\ell}
\beta^*_{nk},
$$
for all $k \in \{0,\ldots,q-1\}$ and $\ell \in \{0,\ldots,s-1\}$, where
$\phi_\ell = 2\pi \ell/s$.

Using NUFFT compute
$$
f_j = \sum_{k=0}^{q-1} \sum_{\ell =0}^{s-1}  e^{\imath x_j \cdot \xi_{k \ell}} a_{i \ell}^*,
$$
with relative error less than $\varepsilon^\text{nuf}$, for $j \in \{1,\ldots,p\}$,
where $\xi_{k \ell} := t_k (\cos \phi_\ell, \sin \phi_\ell)$.
\\[5pt]
\end{algorithm}
Algorithm~\ref{algoB} details the fast method for applying $B$, which consists of applying the adjoint of the operator applied in each step of Algorithm \ref{algo} in reverse order, with slightly different accuracy parameters. Indeed, each step of Algorithm  \ref{algo} consists of applying a linear transform
whose adjoint can be applied in a similar number of operations: the adjoint of the first step (which uses type-2 NUFFT) is a type-1 2-D NUFFT 
\cite{barnett2021aliasing}, the adjoint of the second step (which uses a standard FFT) is an inverse FFT, the adjoint of the third step (fast interpolation, see Remark \ref{rmk:fast_interp}) can be computed by a variety of methods (including NUFFT).

\section{Accuracy guarantees for the fast methods}
\label{sec:proofmainresult}
We state and prove a precise version of the informal result in Theorem \ref{thm1informal}.

\begin{theorem}\label{thm1}
Let $1 \ge \varepsilon > 0$ be given, assume $\lambda \le \sqrt{\pi p}$ and $|\log \varepsilon| \leq \sqrt{p}$. Let $\tilde{B}^*$ and $\tilde{B}$ be operators whose actions consist of applying Algorithm \ref{algo} and Algorithm~\ref{algoB}, respectively. We have
$$
\|\tilde{B}^* f  -  B^* f \|_{\ell^\infty}  \le \varepsilon \|f\|_{\ell^1},
\quad \text{and} \quad
\|\tilde{B} \alpha - B \alpha\|_{\ell^\infty} \le \varepsilon \|\alpha\|_{\ell^1}.
$$
Moreover, both algorithms involve $\mathcal{O}( p \log p + p|\log\varepsilon|^2)$ operations.
\end{theorem}
The proof of Theorem \ref{thm1} is given in Appendix~\ref{sec:proof_main_result}. 
We note that the theorem quantifies the computational accuracy in terms of $\ell^1$-$\ell^\infty$ relative error, which is standard for algorithms involving the NUFFT \cite{barnett2021aliasing,
barnett2019parallel}.
The assumption $|\log \varepsilon| \le \sqrt{p}$ is not restrictive since if $|\log \varepsilon| \ge \sqrt{p}$, then we could directly evaluate $B^* f$ in the same asymptotic complexity $\mathcal{O}(p^2)$.
The proof of Theorem \ref{thm1} relies on the following two key lemmas that estimate a sufficient number of angular nodes and radial nodes in \S \ref{sec:num_angular} and \S \ref{sec:num_nodes}, respectively.

\subsection{Number of angular nodes}\label{sec:num_angular}
Informally speaking, the following lemma shows that $s = \mathcal{O}(\sqrt{p})$ angular nodes are sufficient to achieve error $\gamma$
in the discretization of the integral over $\phi$, see 
\eqref{eq:int_over_phi}.

\begin{lemma} \label{lem:num_angular_nodes}
Let the number of equispaced angular nodes $s$ satisfy
\begin{equation}
 \label{eq:num_angular_nodes}
 \textstyle
s = \lceil \max\{ 7.09 \sqrt{p} , \log_2 \gamma^{-1} \} \rceil.
\end{equation}
Let $x_j = (r_j \cos \theta_j, r_j \sin \theta_j)$. If $\rho \in [\lambda_1,\lambda_m]$, then
$$
\left| \frac{\imath^n}{s} \sum_{\ell=0}^{s-1}  e^{\imath r_j \rho \cos(\theta_j - \phi_\ell)} e^{-\imath n \phi_\ell} - J_n(r_j \rho) e^{-i n \theta_j} \right| \leq \gamma,
$$
for $n \in \{-N_m,\ldots,N_m\}$, and $j \in \{1,\ldots,p\}$, where $\phi_\ell = 2\pi \ell/s$.
\end{lemma}
It will be clear from  the proof that the constant $7.09$ in the statement of the lemma is an overestimate; see Remark \ref{rmk:refined_estimate} for a discussion of how this constant can be improved.

\begin{proof}[Proof of Lemma \ref{lem:num_angular_nodes}]
Let
\begin{equation} \label{eq:g_def}
g_{nj}(\rho,\phi) = \frac{\imath^n}{2\pi} e^{\imath r_j \rho \cos(\theta_j - \phi)} e^{-\imath n \phi},
\end{equation}
for $n \in \{0,\ldots,s-1\}$ and $j \in \{1,\ldots,p\}$. We want to show that
\begin{equation} \label{eq:sum_ang_g}
\left| \frac{2\pi}{s} \sum_{\ell=0}^{s-1} g_{nj}(r,\theta) - J_n(r_j \rho) e^{-i n \theta_j} \right| < \gamma.
\end{equation}
Notice that the sum in \eqref{eq:sum_ang_g}  is a discretization of the integral
\begin{equation}\label{eq:exact_integral}
\int_0^{2\pi} g_{n j}(t_k, \phi) d\phi = J_n(r_jt_k)e^{-\imath n \theta_j},
\end{equation}
where the exact expression for the integral results from \eqref{eq:jn_as_fourier_coeff} and a change of variables from $\phi \mapsto \phi + \theta_j + \pi/2 $ in the integral. It follows from Lemma \ref{newlemma} that
\begin{equation}\label{eq:bound_gj}
\begin{split}
\left| \frac{2\pi}{s}\sum_{\ell=0}^{s-1} g_{n j}(\rho,\phi_\ell) - J_n(r_j \rho)e^{-\imath n \theta_j}\right| &= \left| \frac{2\pi}{s}\sum_{\ell=0}^{s-1} g_{n j}(\rho,\phi_\ell) - \int_0^{2\pi} \!\!\!\! g_j(\rho, \phi) d\phi\right| \\
& \le \frac{{4} \|g_{n j}^{(s)}(\rho,\cdot)\|_{L^1}}{s^s}  ,
\end{split}
\end{equation}
where $g_{n j}^{(s)}(\rho,\phi)$ denotes the $s$-th derivative of 
$g_{n j}(\rho,\phi)$ with respect to $\phi$.
From definition \eqref{eq:g_def} of $g_{nj}(\rho,\phi)$, we have the estimate
$$
\left| g_{nj}^{(s)}(\rho,\phi) \right| \le \frac{(\lambda_m + N_m)^s}{2\pi}, 
$$
for all  $\rho \in [\lambda_1,\lambda_m]$, $n \in \{-N_m,\ldots,N_m\}$, $j=1,\ldots,p$, and $\phi \in [0,2\pi]$.
Therefore, since $4/(2\pi) \le 1$, it suffices to choose $s$ such that
\begin{equation}  \label{eq:angular_error}
\left( \frac{\lambda_m + N_m}{s} \right)^s \le \gamma .
\end{equation}
It follows that choosing
$
s = \max\{ 2 (\lambda_m + N_m), \log_2 \gamma^{-1} 
\}
$
achieves  error at most $\gamma$. To complete the proof, we note that $\lambda_m \le \lambda$, where $\lambda$ is the maximum bandlimit from \S\ref{sec:bandlimit}. Also by \cite[10.21.40]{dlmf} we have
$$
\lambda_{n 1} = n + 1.8575 n^{1/3} + \mathcal{O}(n^{-1/3}),
$$
which implies that the maximum angular frequency 
\begin{equation} \label{eq:bound_N_m}
N_m \le \lambda.
\end{equation}
We conclude that
$
s = \max \{ 4 \lambda, \log_2 \gamma^{-1}  \},
$
is sufficient to achieve error $\gamma$. Since we assume $\lambda \le \sqrt{\pi p}$ and $4 \sqrt{\pi} \le 7.09$, the proof is complete.
\end{proof}

\subsection{Number of radial nodes} \label{sec:num_nodes}
The following lemma shows that $\mathcal{O}(\sqrt{p})$ Chebyshev nodes are sufficient for accurate interpolation in Step~3 of Algorithm \ref{algo}. 
\begin{lemma} \label{lem:num_radial_nodes}
Let the number of radial nodes 
\begin{equation}
q = \lceil \max\left\{ 2.4 \sqrt{p} , \log_2 \gamma^{-1} \right\} \rceil.
\label{eq:num_radial_nodes}
\end{equation}
Let $P_n$ be the degree $q-1$ polynomial such that
$$
P_n(t_k) =  J_n(r_j t_k)e^{-\imath n \theta_j},
$$
for $k \in \{0,\ldots,q-1\}$, 
where $t_k$ are Chebyshev nodes for $[\lambda_1,\lambda_m]$, see \eqref{eq:def_chebyshev}. Then,
$$
|P_n(\rho) -J_n(r_j \rho)e^{-\imath n \theta_j}| \le \gamma,
$$ 
for $\rho \in [\lambda_1,\lambda_m]$, $n \in \{-N_m,\ldots,N_m\}$, and $j \in \{1,\ldots,p\}$.
\end{lemma}
As above, we emphasize that the constant $2.4$ in the statement of this result is an overestimate. 
See \ref{rmk:refined_estimate} for a discussion about how this constant can be improved.
\begin{proof}[Proof of Lemma \ref{lem:num_radial_nodes}]
When interpolating a smooth
differentiable  function $h$ defined on the interval $[a,b]$ using an interpolating polynomial $P$ at $q$
Chebyshev nodes, the residual term $R(\rho) = h(\rho) - P(\rho)$ can be written as
\begin{equation*}
    |R(\rho)| \leq \frac{C_q}{q!} \left(\frac{b-a}{4}\right)^q,
\end{equation*}
where $C_q := \max_{\rho \in [a,b]} |h^{(q)}(\rho) |$; see \cite[Lemma~2.1]{rokhlin1988fast}.
If we apply this result with $[a,b] = [\lambda_{1}, \lambda_{m}]$, the residual satisfies
\begin{equation*}
    |R(\rho)| \leq \frac{C_q}{q !} \left(\frac{\lambda_{m} - \lambda_{1}}{4}\right)^q \leq \frac{C_q }{q !} \left(\frac{ \sqrt{\pi p}}{4}\right)^q ,
\end{equation*}
where the final inequality follows from the bound $\lambda_m \leq  \sqrt{\pi p}$; see \S \ref{sec:bandlimit}. In order to apply this bound to $J_n(r_j\rho)e^{-\imath n \theta_j}$, we estimate 
$$C_q := \max_{\rho \in [\lambda_1,\lambda_m]} \left| \frac{\partial ^ q}{d\rho^q}\left(J_n(r_j\rho)\right)e^{-\imath n \theta_j} \right|.$$
We expand the function $J_n(r_j\rho)$ using the integral identity in \eqref{eq:jn_as_fourier_coeff} and obtain
\begin{eqnarray*}
\left|\frac{\partial ^ q}{d\rho^q}\left(J_n(r_j\rho)\right) \right| &=& 
\left|\frac{1}{2\pi} \int_0^{2\pi}\frac{\partial ^ q}{d\rho^q} \left( e^{\imath r_j\rho \sin(\theta) - \imath n \theta } \right)\text{d}\theta \right| ,\\
&=&  \left|\frac{1}{2\pi} \int_0^{2\pi} \left(\imath r_j \sin(\theta) \right)^q e^{\imath r_j\rho \sin(\theta) - \imath n \theta } \text{d}\theta \right|, \\
&\leq&  \left(\frac{1}{2\pi} \int_0^{2\pi}  \text{d}\theta \right) .
\end{eqnarray*}
In combination with Stirling's approximation \cite[5.11.3]{dlmf}, it follows that
$$
|R(\rho)| \le \frac{1}{ q!} \left( \frac{ \sqrt{\pi p}}{4} \right)^q \leq  \left( \frac{ \sqrt{\pi p} e}{4 q} \right)^q .
$$
Therefore, in order to achieve  error $|R(\rho)| \leq \gamma$, it suffices to set $q$ such that
\begin{equation} \label{eq:nonlinear_q}
\gamma \geq \left( \frac{ \sqrt{\pi p} e}{4 q} \right)^q.
\end{equation}
Setting $\sqrt{\pi p} e/4 q= 1/2$
and solving for $q$ gives
$$
q = \frac{\sqrt{\pi} e \sqrt{p}}{2} \approx 2.4 \sqrt{p}
\quad
\implies
\quad
q \ge \max\{ 2.4 \sqrt{p} , \log_2 \gamma^{-1}),
$$
is sufficient to achieve  error less than $\gamma$. 
\end{proof}

\begin{remark}[Improving estimates for  number of radial and angular nodes]
\label{rmk:refined_estimate}
 While Lemmas \ref{lem:num_radial_nodes} and \ref{lem:num_angular_nodes} show that the number of radial nodes 
$q$ and angular nodes $s$ are $\mathcal{O}(\sqrt{p})$, the constants in the lemmas are not optimal. 
For practical purposes, choosing the minimal number of nodes possible to achieve the desired error is advantageous 
to improve the run time constant of the algorithm, and it is clear from
the proofs how the estimates can be refined.
For Lemma \ref{lem:num_radial_nodes} we set
 $Q = \lceil 2.4 \sqrt{p} \rceil$, and motivated by \eqref{eq:nonlinear_q} compute
$$
\gamma^\text{rad}(q) = \frac{1}{\sqrt{\pi} q!} \left( \frac{\sqrt{\pi p}}{4} \right)^q,
$$
for $q = 1,\ldots,Q$ and choose the smallest value $q^*$ of $q$ such that $\gamma^\text{rad}(q^*) \le \gamma$. Similarly, for Lemma \ref{lem:num_angular_nodes}, we set
$S = \lceil 7.09 \sqrt{p} \rceil$, and motivated by 
\eqref{eq:angular_error} compute
$$
\gamma^\text{ang}(s) =  \left( \frac{\lambda_m + N_m}{s} \right)^s,
$$
for $s = 1,\ldots,S$ and choose the smallest value $s^*$ of $s$ such that $\gamma^\text{ang}(s) \le \gamma$. Then, it follows  that $2.4 \sqrt{p}$ and $7.09 \sqrt{p}$ can be replaced by $q^*$ and $s^*$, 
 in the statements of Lemmas
  \ref{lem:num_radial_nodes} and \ref{lem:num_angular_nodes}, respectively. This procedure improves the estimate of the required number of angular and radial nodes by a constant factor.
\end{remark}

\section{Numerical results} \label{sec:numerics}

\begin{remark}[FFT Bandlimit heuristic] \label{rmk:bandlimit_heuristic}
One heuristic for setting the bandlimit is based on the fast Fourier transform (FFT). For a centered FFT on a signal of length $L$, the maximum frequency is $\pi^2 (L/2)^2$, which corresponds to a bandlimit of $\lambda = \pi L/2$. Note that 
\begin{equation} \label{eq:bandlimit_est}
\pi L/2  \approx 1.57 L < 
1.77 L \approx
\sqrt{\pi} \lfloor (L-1)/2 \rfloor /2, 
\end{equation}
so this FFT bandlimit heuristic does indeed produce a reasonable bandlimit below the bound \eqref{eq:bandlimit} derived from Weyl's law. We use this bandlimit for our numerical experiments. The computational complexity and accuracy guarantees of the method presented in this paper hold for any bandlimit $\lambda = \mathcal{O}(L)$. However, the fact that the fast method performs interpolation in Fourier space inside a disk bounded by the maximum bandlimit provides additional motivation for this FFT-based heuristic since it will ensure that the disk will be contained within the square in frequency space used by the two-dimensional FFT. 
\end{remark}

\subsection{ Numerical accuracy results}
In this section, we report numerical results for the accuracy of our FDHT method compared to matrix
multiplication. The implementation of the method is based on the parameters $\varepsilon^{\text{dis}}$, $\varepsilon^{\text{nuf}}$, $\varepsilon^{\text{fst}}$, $s$, and $q$, which result in the error guarantees in Theorem~\ref{thm1}. However, since these theoretical error bounds are slightly pessimistic, and do not account for errors from finite precision arithmetic, the parameters used by the implementation of the algorithm are tuned slightly so that the code achieves the desired accuracy in numerical tests. Recall that $B : \mathbb{C}^m \rightarrow \mathbb{C}^p$ maps coefficients to images by
$$
(B \alpha)_j = \sum_{i=1}^m \alpha_i \psi_i(x_j)h,
$$
and its adjoint
transform $B^* : \mathbb{C}^p \rightarrow \mathbb{C}^m$ maps images to coefficients  by
$$
(B^* f)_i =\sum_{j=1}^p f_j \overline{\psi_i(x_j)}h,
$$
see \S \ref{sec:main_result}. By defining the $m \times p$ matrix $B$ by
$$
B_{ij} = \psi_i(x_j) h,
$$
we can apply $B$ and $B^*$ by dense matrix multiplication to test the accuracy of our fast method. Since the size of the matrix  scales like $L^4$ for $L \times L$ images, constructing these matrices quickly becomes
prohibitive so the comparison is only given up to $L=160$, see Table \ref{tab:accuracy}, where
$$
\text{err}_\alpha = \frac{\|\alpha_\text{fast} - \alpha_\text{dense}\|_{\ell^2}}{\|\alpha_\text{dense}\|_{\ell^2}} \quad \text{and} \quad
\text{err}_f = \frac{\|f_\text{fast} - f_\text{dense}\|_{\ell^2}}{\|f_\text{dense}\|_{\ell^2}},
$$
denote the relative errors of the coefficients  and the image, respectively,
where $\alpha_\text{dense} = B^* f$ 
and $f_\text{dense} = B \alpha$ are computed by dense matrix multiplication and $\alpha_\text{fast}$ 
and $f_\text{fast}$ are the corresponding quantities computed using the fast algorithm of this paper.

\begin{table}[h!]
\centering
\caption{Relative error of fast method compared to dense matrix multiplication.} \label{tab:accuracy}
\begin{tabular}{r|ccc || ccc}
$L$ & $\varepsilon$ & $\text{err}_\alpha$ & $\text{err}_f$ & $\varepsilon$ & $\text{err}_\alpha$ & $\text{err}_f$ \\
\hline
64 &  1.00e-04 &  1.92422e-05 &  2.10862e-05 &  1.00e-10 &  3.55320e-11 &  2.36873e-11 \\
96 &  1.00e-04 &  1.82062e-05 &  2.52219e-05 &  1.00e-10 &  2.99849e-11 &  2.48166e-11 \\
128 &  1.00e-04 &  1.90648e-05 &  2.41142e-05 &  1.00e-10 &  3.25650e-11 &  2.61890e-11 \\
160 &  1.00e-04 &  2.00748e-05 &  2.49488e-05 &  1.00e-10 &  3.13903e-11 &  3.50455e-11 \\
\hline \hline
64 &  1.00e-07 &  2.03272e-08 &  2.98083e-08 &  1.00e-14 &  7.41374e-15 &  6.82660e-15 \\
96 &  1.00e-07 &  2.28480e-08 &  2.58272e-08 &  1.00e-14 &  9.82890e-15 &  8.80843e-15 \\
128 &  1.00e-07 &  2.69215e-08 &  2.27676e-08 &  1.00e-14 &  1.21146e-14 &  1.11909e-14 \\
160 &  1.00e-07 &  2.47053e-08 &  2.51146e-08 &  1.00e-14 &  1.36735e-14 &  1.51430e-14 \\

\end{tabular}
\end{table}

The image used for the accuracy comparison is a tomographic projection of a 3-D density map representing a  bio-molecule (E. coli 70S ribosome) \cite{shaikh2008spider}, retrieved from the online EM data bank~\cite{lawson2016emdatabank}.

\subsection{Timing results}
In this section, we plot the timing of our FDHT method for $L \times L$ images with
$p = L^2$ pixels.
We demonstrate that the method does indeed have complexity $\mathcal{O}(p \log
p)$ and that the timings are practical. We plot the time of pre-computation and the time of applying $B$ using the fast method; for comparison, we include timings for forming and applying the dense matrix $B$, see Fig. \ref{fig:03}. The timings for applying $B^*$ are similar to the timings for applying $B$ (since the algorithm consists of applying similar transforms in the reverse order), so a separate plot was not included.

The timings were carried out on a computer with an AMD 5600X
processor and 24 GB of memory. We set $\varepsilon = 10^{-7}$ for the reported timings, and compare to the dense method up to $L=160$. For $L>160$, comparison to the dense method was prohibitively expensive. For reference, storing the dense transform matrix in double precision complex numbers for $L=512$ would require about $640$ GB of memory.
The NUFFT uses the FINUFFT implementation \cite{barnett2021aliasing,barnett2019parallel}. The image used for the timing results is a tomographic projection of a 3-D density map representing the SARS-CoV-2 Omicron spike glycoprotein complex \cite{guo2022structures}, retrieved from the online EM data bank~\cite{lawson2016emdatabank}.
\begin{figure}[ht!]
\centering
\begin{tabular}{cc}
\includegraphics[width=.45\textwidth]{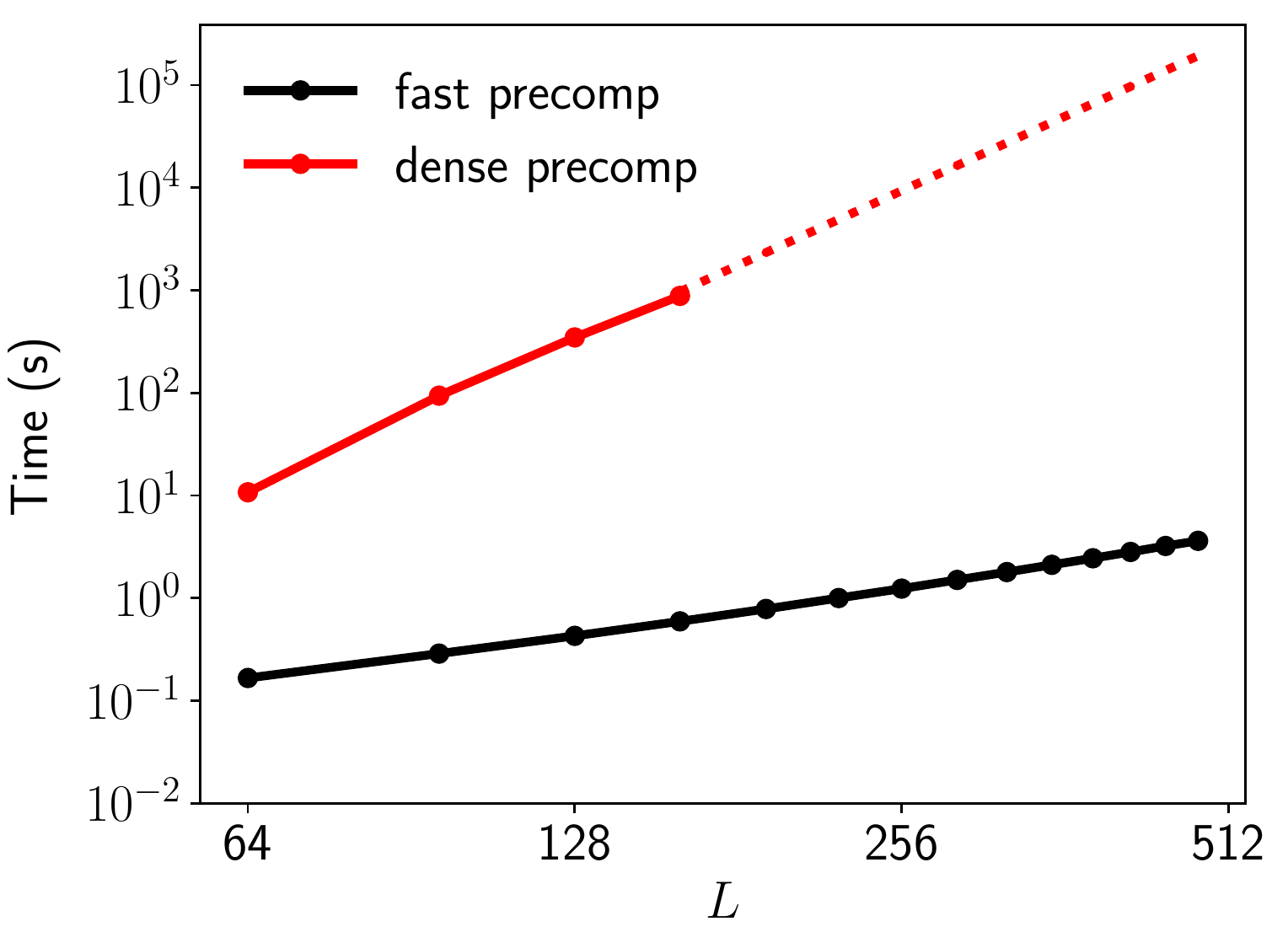} &
\includegraphics[width=.45\textwidth]{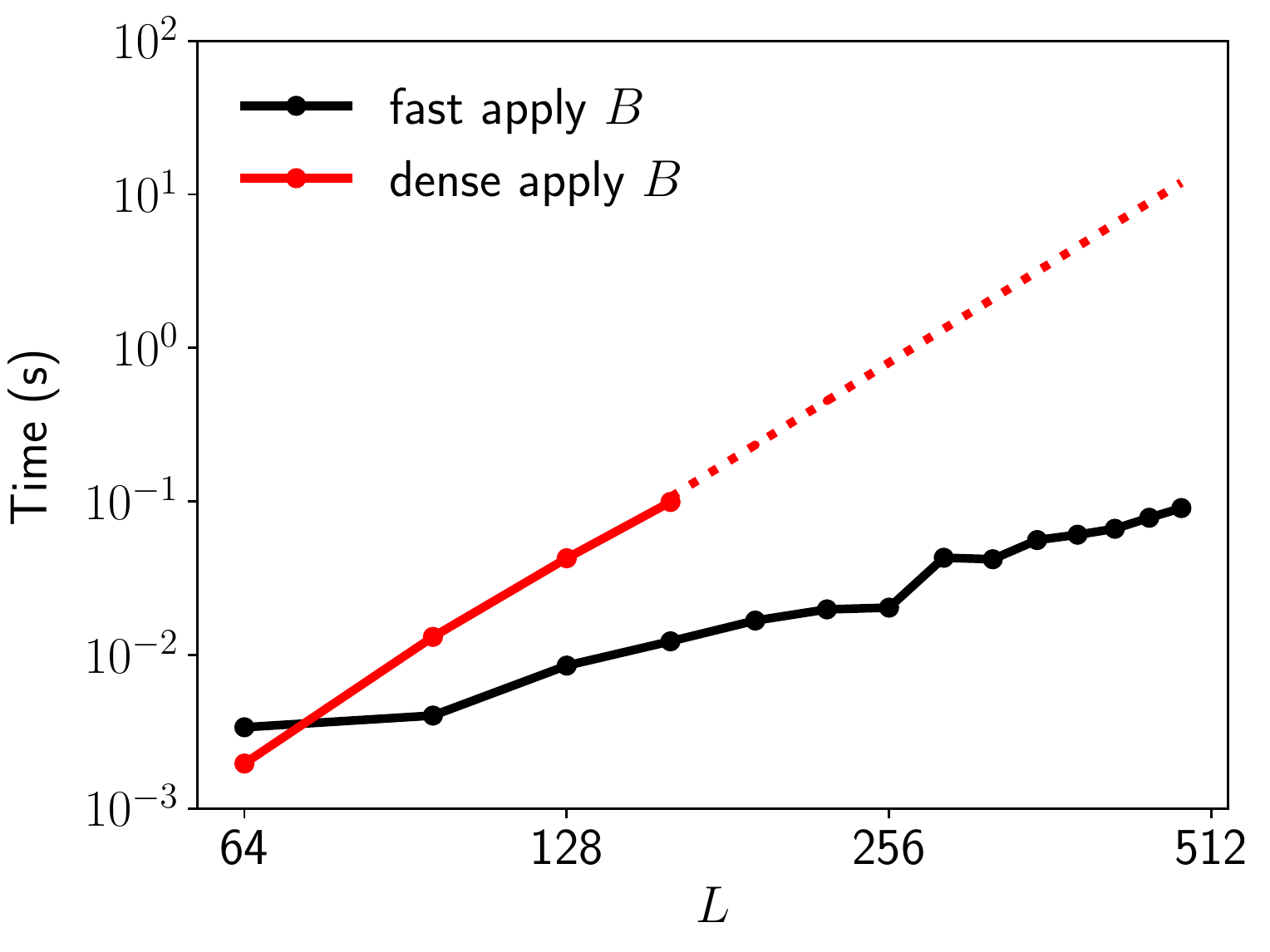}
\end{tabular}
\caption{Timings of fast method versus dense method for pre-computation (left) and applying $B$ (right). 
The timings for the dense method for $L > 160$ are extrapolated since the memory requirements for the dense method were prohibitive. 
}  \label{fig:03}
\end{figure}

\begin{remark}[Pre-computation time negligible when transforming many images]
The pre-computation involves organizing Bessel function roots and creating data structures for the NUFFT and interpolation steps of the algorithm. The pre-computation only needs to be performed once for a given size of image $L$ and becomes negligible when the method is used to expand a large enough set of images (around $100$ images), which is a typical use-case in, for example, applications in cryo-EM \cite{bhamre2016denoising}.
\end{remark}

\begin{remark}[Breakdown of timing of fast algorithm]
Each step of the algorithm has roughly the same magnitude.
For example, for $L=512$ and $\varepsilon = 10^{-7}$ the timings of the NUFFT, FFT, and Interpolation steps of the algorithm for applying $B$ are $0.035$,  $0.046$, and $0.026$ seconds, respectively. We note that the timing of each step is dependent on the choice of parameters. For example, sampling more points will increase the cost of the NUFFT step but decrease the cost of the interpolation step since sparser interpolation matrices can be used; decreasing $\varepsilon$ will increase the cost of the NUFFT step.
\end{remark}

\begin{remark}[Parallelization]
The timings reported in Fig.~\ref{fig:03} are for a single-threaded CPU code. However, each step of the code is amenable to parallelization through GPU implementations. Indeed, the NUFFT step has a GPU implementation \cite{shih2021cufinufft}, and the 2-D FFT and interpolation steps can also benefit from straightforward parallelization schemes.
\end{remark}

\subsection{Numerical example: convolution and rotation}
We lastly present an example illustrating the use of the steerable and fast radial convolution properties of the eigenbasis. The example is motivated by cryo-EM, wherein tomographic projection images of biological molecules in a sample are registered by electron beams; see, for example, \cite{frank2006three} for more information. Because of aberrations within the electron-microscope and random in-plane rotations of the molecular samples, the registered image $I_r$ does not precisely coincide with the actual projection image $I_p$, and the following model is used:
\begin{equation}
    I_r(x) = c(|x|)*R_{\theta}\left( I_p(x) \right) + \eta,
\end{equation}
where $R_\theta$ describes rotation around the origin by an angle of $\theta$, $c$ is a radial function termed the point-spread function and $\eta$ is additive white noise. The function $\widehat{c}$ is, in turn, known as the contrast transfer function (CTF). Examples of point spread functions are shown in Fig. \ref{fig:04}.

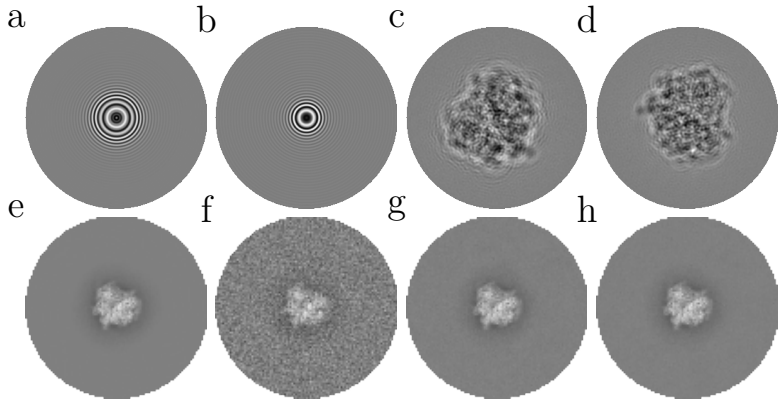
\begin{figure}[ht]
\centering
\begin{tikzpicture}[scale=\textwidth/10cm]
\node[anchor=north west] at (0,2) {\includegraphics[width=2.4cm]{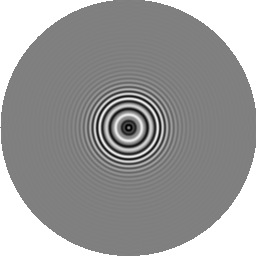}};
\node[anchor=north west] at (2,2) {\includegraphics[width=2.4cm]{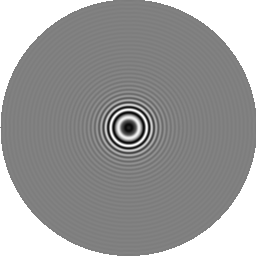}};
\node[anchor=north west] at (4,2) {\includegraphics[width=2.4cm]{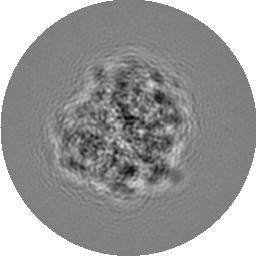}};
\node[anchor=north west] at (6,2) {\includegraphics[width=2.4cm]{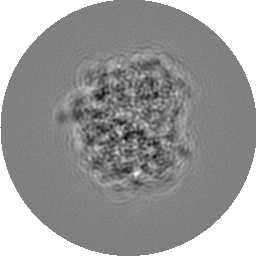}};
\node[anchor=north west] at (0,0) {\includegraphics[width=2.4cm]{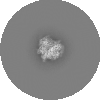}};
\node[anchor=north west] at (2,0) {\includegraphics[width=2.4cm]{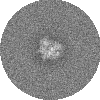}};
\node[anchor=north west] at (4,0) {\includegraphics[width=2.4cm]{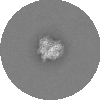}};
\node[anchor=north west] at (6,0) {\includegraphics[width=2.4cm]{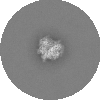}};
\node at (0,2) {\LARGE a};
\node at (2,2) {\LARGE b};
\node at (4,2) {\LARGE c};
\node at (6,2) {\LARGE d};
\node at (0,0) {\LARGE e};
\node at (2,0) {\LARGE f};
\node at (4,0) {\LARGE g};
\node at (6,0) {\LARGE h};
\end{tikzpicture}
\caption{Two different point spread functions (a--b), result of their convolution with a fixed image and subsequent rotation (c--d), (e) Projection of reference image into the eigenbasis using the fast algorithm. (f-g) result of deconvolution algorithm using $t=1,$ $3,5$, respectively. 
}  \label{fig:04}
\end{figure}

Notably, the regions of the frequency space where $\widehat{c}$ equals zero destroy information of $I_p(x)$. However, the fact that convolution is a diagonal transformation of the coefficients in the basis of eigenfunctions enables reconstruction of a fixed projection image $I_p$ from a small number of registered images $I_{r}^{(i)}$ with different point spread functions $c_i(|x|)$, rotations $R_{\theta_i}$, and noise $\eta^{(i)}$, for $i=1, \ldots , t$. From Lemma~\ref{lem:coeff_from_beta}, it follows that the basis coefficients $\alpha_{nk}^{(i)}$ of the registered images satisfy
\begin{equation}\label{eq:deconv_eq}
    \alpha_{nk}^{(i)} = \widehat{c}_i(\lambda_{nk}) e^{\imath n \theta_i} \alpha_{nk}^{(0)} + \eta_{nk}^{(i)}, \quad \text{ for } i = 1, \ldots , t,
\end{equation}
where $\alpha_{nk}^{(0)}$ denote the basis coefficients of $I_p$. We assume that the parameters $\theta_i$ and $c_i$ are known or estimated to a desired precision. We remark that the standard FFT can be used to solve this problem when there are no rotations.

To recover the $\alpha_{nk}^{(0)}$, we find the least-squares optimizers of \eqref{eq:deconv_eq}. To improve the conditioning of the problem, \eqref{eq:deconv_eq} is thresholded to exclude the values of $i$ for which $\widehat{c}_i(\lambda_{nk})$ has sufficiently low magnitude. We therefore estimate $\alpha^{(0)}_{nk}$ by $\alpha^{(0)}_{nk} \approx \alpha_{nk}$, with $\alpha_{nk}$ defined by
\begin{equation}\label{eq:est_alpha_wiener}
\alpha_{nk} = \min_{\alpha_{nk}} \sum_{(n,k) \in \mathcal{I}} \sum_{i=1}^t \gamma_{nk}^{(i)}\cdot \left| \alpha_{nk}^{(i)} - \widehat{c}_i(\lambda_{nk}) e^{\imath n \theta_i} \alpha_{nk} \right|^2,
\end{equation}
where $\gamma_{nk}^{(i)} = 0$ if $|\widehat{c}_i(\lambda_{nk})| < \tau$, for a given threshold $\tau$, and $\gamma_{nk}^{(i)} = 1$ otherwise. This describes a decoupled least-squares problem for each coefficient $\alpha_{nk}$, which can be solved efficiently. We remark that \eqref{eq:est_alpha_wiener} is a basic version of Wiener filtering \cite{bhamre2016denoising}, which we use for simplicity of exposition. The result of this procedure for different values of $t$ and a non-zero value of the noise $\eta$ is shown in Fig.~\ref{fig:04}.

\section{Discussion} \label{sec:discuss}
This paper presents a fast method for expanding a set of $L\times L$-images into the basis of eigenfunctions of the Laplacian on the disk. The approach calculates the expansion coefficients from interpolation of the Fourier-transform of the image on distinguished subsets of the frequency space and relies on an integral identity of the Fourier-transform of the eigenfunctions. Unlike previous approaches \cite{MR3472531}, we demonstrate that our fast method is guaranteed to coincide with a dense, equivalent method up to a user-specified precision. 
Moreover, our method provides a natural way to compute the convolution with radial functions.
Potential extensions of the presented method include extending the method to three dimensions or other domains in two dimensions.

\subsection*{Acknowledgements}
The authors would like to thank Joakim And\'en, Yunpeng Shi, and Gregory Chirikjian for their helpful comments on a draft of this paper. We also thank two anonymous reviewers for their comments which improved the exposition of the manuscript.

\bibliographystyle{plain}

\bibliography{references}

\appendix
\section{Proof of Theorem \ref{thm1}}\label{sec:proof_main_result}
This section proves Theorem~\ref{thm1}.

\subsection{Proof of accuracy of Algorithm \ref{algo}}
Let $\tilde{\alpha}_i$ be the output of Algorithm \ref{algo}, including the
error from the NUFFT and fast interpolation steps. By composing the steps of
the algorithm, we have 
$$
\tilde{\alpha}_i =  c_i h \left( \sum_{k=0}^{q-1} \left( \frac{\imath^{n_i}}{s}
\sum_{\ell = 0 }^{s-1} \left( \sum_{j=1}^p f_j e^{-\imath x_j \cdot \xi_{k
\ell}} + \delta^\text{nuf}_{k \ell} \right)e^{- \imath n_i \phi_\ell} \right)
u_k(\lambda_i) +  \delta^\text{fst}_i\right),
$$
where $\delta^\text{nuf}_{k \ell}$ and $\delta^{fst}_i$ denote the error from
the NUFFT and fast interpolation, respectively. These satisfy
$\ell^1$-$\ell^\infty$ relative error bounds
\begin{equation} \label{err:nuf}
\|\delta^\text{nuf}\|_{\ell^\infty} 
\le \varepsilon^\text{nuf} \sum_{j=1}^p |f_j e^{-\imath x_j \cdot \xi_{k \ell}}| = \varepsilon^\text{nuf} \|f\|_{\ell^1},
\end{equation}
and (using $\varepsilon^\text{nuf} \le 1$ which we can ensure holds) 
\begin{equation} \label{err:fst}
\|\delta^\text{fst}\|_{\ell^\infty} 
\le \varepsilon^\text{fst} \sum_{k=0}^{q-1} \left| \frac{\imath^{n_i}}{s}
\sum_{\ell=0}^{s-1} \left( \sum_{j=1}^p f_j e^{-\imath x_j \cdot \xi_{k \ell}}
+ \delta_{k \ell}^{\text{nuf}} \right) e^{-\imath n_i \phi_\ell} \right|
\le
2 \varepsilon^\text{fst} q \|f\|_{\ell^1},
\end{equation}
where $\varepsilon^\text{nuf}$ and $\varepsilon^\text{fst}$ are the relative
error parameters for the NUFFT and fast interpolation, respectively. 
Let
$\alpha_i$ denote the output of Algorithm \ref{algo} without the NUFFT and fast
interpolation error terms, i.e.,
\begin{equation}\label{eq:def_alpha_in_proof}
\alpha_i =  c_i h \sum_{k=0}^{q-1} \frac{\imath^{n_i}}{s}
\sum_{\ell = 0 }^{s-1}  \sum_{j=1}^p f_j e^{-\imath x_j \cdot \xi_{k
\ell}}  e^{- \imath n_i \phi_\ell}
u_k(\lambda_i).
\end{equation}
We have
\begin{equation}
\begin{split}
|\alpha_i - \tilde{\alpha}_i| &\le  c_i h \left( \sum_{k=0}^{q-1} \frac{1}{s} \sum_{\ell=0}^{s-1} \|\delta^\text{nuf}\|_{\ell^\infty} |u_k(\lambda_i)| + \|\delta^\text{fst}\|_{\ell^\infty}\right) \\
&\le c_i h \left( \|\delta^\text{nuf}\|_{\ell^\infty} \sum_{k=0}^{q-1} |u_k(\lambda_i)| + \|\delta^\text{fst}\|_{\ell^\infty} \right) \\
&\le  2 \sqrt{2} \left( \|\delta^\text{nuf}\|_{\ell^\infty} (2 + \frac{\pi}{2}
\log q) + \|\delta^\text{fst}\|_{\ell^\infty} \right), 
\end{split}
\end{equation}
where the final inequality follows from Lemma \ref{lemma:chebyshev_perturbation} and the 
fact that
$$
c_ih \leq 2\sqrt{2},
$$
which follows from the definition of $h$ in \eqref{eq:pixel_locs} and Lemma \ref{lem:cj}. Combining this equality with \eqref{err:nuf} and
\eqref{err:fst} gives
\begin{equation} \label{nuf+fst}
\|\alpha - \tilde{\alpha}\|_{\ell^\infty} \le 2 \sqrt{2} \left( \varepsilon^\text{nuf} \left(2 + \frac{\pi}{2} \log q \right) + 2 q \varepsilon^\text{fst} \right) \|f\|_{\ell^1}.
\end{equation}
Setting 
\begin{equation} \label{eq:nufst}
\varepsilon^\text{nuf} = (2 \sqrt{2}(2 + \frac{\pi}{2} \log q))^{-1}
(\varepsilon/4) \quad \text{and} \quad \varepsilon^\text{fst} = (4
\sqrt{2} q)^{-1} (\varepsilon/4),
\end{equation}
gives
\begin{equation} \label{errp1}
\|\alpha - \tilde{\alpha}\|_{\ell^\infty} \le \frac{\varepsilon}{2} \|f\|_{\ell^1}.
\end{equation}
By the definition of $\alpha_i$ in \eqref{eq:def_alpha_in_proof}, we have
\begin{equation}
\begin{split}
\alpha_i &= c_i h \sum_{j=1}^p f_j \left( \sum_{k=0}^{q-1} \frac{\imath^{n_i}}{s}
\sum_{\ell = 0 }^{s-1} e^{-\imath x_j \cdot \xi_{k \ell}}  e^{- \imath n_i
\phi_\ell} u_k(\lambda_i)  \right) \\
&= c_i h \sum_{j=1}^p f_j \left(J_{n_i}(r_j \lambda_i) e^{-\imath n_i
\theta_j} + \delta_{i j}^\text{dis} \right) = (B^* f)_i + c_i h \sum_{j=1}^p f_j \delta_{i j}^\text{dis},
\end{split}
\end{equation} 
where the third equality follows from Lemma \ref{lem:dis} with $\delta^\text{dis}_{i j}$ a discretization error that is bounded by $\|\delta^\text{dis}\|_{\ell^\infty} \le
(3 + \frac{\pi}{2} \log q)
\varepsilon^\text{dis}$. It follows that
$$
|\alpha_i - (B^*f)_i | \le c_i h \|f\|_{\ell^1} \|\delta^\text{dis}\|_{\ell^\infty} \le 2 \sqrt{2} 
(3 + \frac{\pi}{2} \log q)
\varepsilon^\text{dis} \|f\|_{\ell^1}.
$$
Setting 
\begin{equation} \label{eq:epsdis}
\varepsilon^\text{dis} = (2 \sqrt{2}(3 + \frac{\pi}{2} \log{(2.4\sqrt{p}) }))^{-1} \frac{\varepsilon}{2},
\end{equation}
and combing with \eqref{errp1} gives
$$
\|\alpha - B^* f \|_{\ell^\infty} \le \varepsilon \|f\|_{\ell^1} ,
$$
which completes the proof of the accuracy guarantees for Algorithm \ref{algo}.

\subsection{Proof of accuracy of Algorithm \ref{algoB}}

Let $\tilde{f}$ be the output of Algorithm \ref{algoB}, including the
approximation error from using fast interpolation and the NUFFT. By composing the steps of Algorithm \ref{algoB} we have
$$
\tilde{f}_j = \sum_{k=0}^{q-1} \sum_{\ell=0}^{s-1} \left( \sum_{n = -N_m}^{N_m} \left( \sum_{i: n_i = n} u_k(\lambda_i) c_i h \alpha_i + \delta_{nk}^\text{fst} \right) \frac{(-\imath)^n}{s} e^{\imath n \phi_\ell} \right)
 e^{-\imath x_j \cdot \xi_{k \ell}} + \delta_j^\text{nuf},
$$
where $\delta_{nk}^\text{fst}$ and $\delta_j^\text{nuf}$ denote the error from
the fast interpolation and NUFFT, respectively, which satisfy
$\ell^1$-$\ell^\infty$ relative error bounds. We have
$$
\|\delta^\text{fst}_{n}\|_{\ell^\infty} \le \varepsilon^\text{fst} \sum_{i :
n_i = n} c_i h |\alpha_i| \le 2 \sqrt{2} \varepsilon^\text{fst} \sum_{i : n_i =
n} |\alpha_i|,
$$
where $\delta_n^\text{fst} = (\delta_{nk}^\text{fst})_{k =0}^{q-1}$, and
\begin{equation}
\begin{split}
\|\delta^\text{nuf}\|_{\ell^\infty} &\le \varepsilon^\text{nuf}
\sum_{k=0}^{q-1} \sum_{\ell=0}^{s-1} \frac{1}{s} \left| \sum_{n = -N_m}^{N_m}
\left( \sum_{i : n_i = n} u_k(\lambda_i) c_i h \alpha_i +
\delta_{nk}^\text{fst} \right) \right| \\ 
& \le  \varepsilon^\text{nuf}  \sum_{i=1}^m
\left( \sum_{k=0}^{q-1} 
|u_k(\lambda_i)| \right) c_i h |\alpha_i |  + \varepsilon^\text{nuf} \sum_{k=0}^{q-1} \sum_{n=-N_m}^{N_m} \!\!\!\!\! 2 \sqrt{2} \varepsilon^\text{fst} \!\! \sum_{i : n_i = n} \!\!|\alpha_i| \\
&\le \varepsilon^\text{nuf} \left( \left( 2 + \frac{\pi}{2} \log q \right) 2 \sqrt{2} \|\alpha\|_{\ell^1} +  q 2 \sqrt{2} \varepsilon^\text{fst} \|\alpha\|_{\ell^1} \right) \\
&\le  \varepsilon^\text{nuf} 2 \sqrt{2} \left( \left( 2 + \frac{\pi}{2} \log q \right)  +  1\right) \|\alpha\|_{\ell^1},
\end{split}
\end{equation}
where the final inequality assumes $q  \varepsilon^\text{fst} \le 1$.
Let $f_j$ denote the output of Algorithm \ref{algoB}, ignoring the error from the fast interpolation and NUFFT, i.e.,
\begin{equation}\label{eq:def_f_in_proof}
f_j = \sum_{k=0}^{q-1} \sum_{\ell=0}^{s-1}  \sum_{n = -N_m}^{N_m}  \sum_{i: n_i = n} u_k(\lambda_i) c_i h \alpha_i  \frac{(-\imath)^n}{s} e^{\imath n \phi_\ell} 
 e^{-\imath x_j \cdot \xi_{k \ell}} .
\end{equation}
We have
\begin{equation}
\begin{split}
|f_j - \tilde{f}_j| &\le \left( \sum_{k=0}^{q-1} \sum_{\ell=0}^{s-1}
\frac{1}{s} \sum_{n=-N_m}^{N_m} |\delta_{nk}^\text{fst}| \right) +
|\delta_j^\text{nuf}| \\ 
&\le \left( \sum_{k=0}^{q-1} 2 \sqrt{2} \varepsilon^\text{fst} \|\alpha\|_{\ell^1} \right) +
\varepsilon^\text{nuf} 2 \sqrt{2} \left( \left( 2 + \frac{\pi}{2} \log q
\right)  +  1 \right) \|\alpha\|_{\ell^1} \\
&\le
\left( \varepsilon^\text{fst} 2 \sqrt{2} q +
\varepsilon^\text{nuf} 2 \sqrt{2}  \left( 3 + \frac{\pi}{2} \log q
\right)    \right) \|\alpha\|_{\ell^1} .
\end{split}
\end{equation}
Setting 
\begin{equation}\label{eq:error_nufft_B_proof}
\varepsilon^\text{nuf} = (2 \sqrt{2} (3 + \frac{\pi}{2} \log q))^{-1} (\varepsilon/4)  \quad \text{and} \quad \varepsilon^\text{fst} = (2 \sqrt{2} q)^{-1} (\varepsilon/4)
\end{equation}
gives
$$
\|f - \tilde{f}\|_{\ell^\infty} \le \frac{\varepsilon}{2} \|\alpha\|_{\ell^1}.
$$
By the definition of $f_j$ in \eqref{eq:def_f_in_proof}, we have
\begin{equation}
\begin{split}
f_j 
&= \sum_{k=0}^{q-1} \sum_{\ell=0}^{s-1} \sum_{i=1}^m u_k(\lambda_i) c_i h \alpha_i \frac{(-\imath)^n}{s} e^{\imath n \phi_\ell} 
 e^{-\imath x_j \cdot \xi_{k
\ell}}
\\
&= \sum_{i=1}^m c_i h \alpha_i \left( \sum_{k=0}^{q-1} \frac{(-\imath)^n}{s}
\sum_{\ell=0}^{s-1} u_k(\lambda_i)  e^{\imath n \phi_\ell} e^{-\imath x_j \cdot
\xi_{k \ell}}
\right) \\
&= \sum_{i=1}^m c_i h \alpha_i \left( J_{n_i} (r_j \lambda_i) e^{\imath n_i \theta_j} + \bar{\delta}^\text{dis}_{i j} \right) = (B f)_j + \sum_{i=1}^m c_i h \alpha_i \bar{\delta}^\text{dis}_{i j},
\end{split}
\end{equation}
where the third equality follows from Lemma \ref{lem:dis} with $\bar{\delta}^\text{dis}_{ij}$ a discretization error that satisfies
$\|\bar{\delta}^\text{dis}\|_{\ell^\infty} \le (3 +\frac{\pi}{2} \log q)\varepsilon^\text{dis}$ . It follows that
$$
\|f - B \alpha\|_{\ell^\infty} \le 2 \sqrt{2} (3 \frac{\pi}{2} \log q) \varepsilon^\text{dis} \|\alpha
\|_{\ell^\infty} .
$$
Setting 
\begin{equation} \label{eq:epsdis2}
\varepsilon^\text{dis} =  (2 \sqrt{2}(3 + \frac{\pi}{2} \log{(2.4\sqrt{p})}))^{-1} \frac{\varepsilon}{2}
\end{equation}
and combining with \eqref{errp1} gives
$$
\|f - B \alpha \|_{\ell^\infty} \le \varepsilon \|\alpha\|_{\ell^1} ,
$$
which completes the proof of the accuracy guarantees for Algorithm \ref{algoB}. \qed

\subsection{Proof of computational complexity of Algorithm \ref{algo} and \ref{algoB}. }

\subsubsection{Computational complexity of NUFFT} 
Both Algorithms \ref{algo} and \ref{algoB} use the same number of source points and target points and have asymptotically similar error parameters, and thus have the same computational complexity. In both cases,
the number of source points is $p$, the number of target points is $s q = \mathcal{O}(p)$ (see the definition of $s$ and $q$ in 
Lemmas \ref{lem:num_angular_nodes} and \ref{lem:num_radial_nodes}), and the error parameter $\varepsilon^\text{nuf} = \mathcal{O}(\varepsilon/\log q)$, see \eqref{eq:nufst} and \eqref{eq:error_nufft_B_proof}, it follows that the computational complexity is $\mathcal{O} (p\log p + p\left|\log \varepsilon - \log \log q \right|^2 )$, see \eqref{eq:nufft}
or \cite{barnett2021aliasing, barnett2019parallel}, which simplifies to
$\mathcal{O}(p\log p + p\left|\log \varepsilon \right|^2)$
operations.

\subsubsection{Computational complexity of FFT} 
Algorithms \ref{algo} and \ref{algoB} use an FFT and inverse FFT (which both have the same computational complexity) on a similar amount of data. In particular, they perform $\mathcal{O}(\sqrt{p})$ applications of the FFT of size $\mathcal{O}(\sqrt{p})$. Therefore, the computational complexities are $\mathcal{O} (p\log p)$.

\subsubsection{Computational complexity of fast interpolation}

There are a number of ways to perform fast polynomial interpolation, see Remark \ref{rmk:fast_interp}. For consistency with the rest of the paper, assume that fast interpolation is performed using the NUFFT whose computational complexity in dimension $d$ is stated in \eqref{eq:nufft}. 

Recall that $N_m = \max\{ n_j \in \mathbb{Z} : j \in \{1,\ldots,m\} \}$ and $K_n = \max \{ k \in \mathbb{Z}_{>0} : \lambda_{n k} \le \lambda \text{ for some } n \in \mathbb{Z}\}$. By \eqref{eq:bound_N_m} we have $N_m \le \sqrt{\pi p}.$ Fix $n \in \{-N_m,\ldots,N_m\}$, we need to compute a polynomial interpolation from $q = \mathcal{O}(\sqrt{p})$ source points (Chebyshev nodes) to $K_n$ target nodes. The computational complexity of each interpolation problem to $\ell^1$-$\ell^\infty$ relative error $\delta$ is $\mathcal{O}(\sqrt{p} \log p + K_n |\log \delta|)$. Summing over the $2N_m+1 = \mathcal{O}(\sqrt{p})$ interpolation problems gives a total complexity of $\mathcal{O}(p \log p + p |\log \delta|)$, where we used the fact that $\sum_{n=-N_m}^{N_m} K_n = m = \mathcal{O}(p)$.
It follows from \eqref{eq:nufst} and \eqref{eq:error_nufft_B_proof} that the computational complexities are $\mathcal{O}(p \log p + p |\log \varepsilon|)$. 

\subsubsection{Summary} Since all of the steps are $\mathcal{O}(p\log p + p |\log \varepsilon|^2)$ the proof is complete.
\qed

\subsection{Technical Lemmas}\label{sec:technical_lemmas}
We first state and prove a lemma that combines 
Lemma \ref{lem:num_angular_nodes} and Lemma \ref{lem:num_radial_nodes}.
\begin{lemma} \label{lem:dis}
Let $s$ and $q$ be defined by Lemma \ref{lem:num_angular_nodes} and  \ref{lem:num_radial_nodes} with accuracy parameter $\gamma > 0$. Then
$$
\left| \sum_{k=0}^{q-1} \frac{\imath^{n_i}}{s}
\sum_{\ell = 0 }^{s-1} e^{-\imath x_j \cdot \xi_{k \ell}}  e^{- \imath n_i
\phi_\ell} u_k(\lambda_i) - J_n(r_j \lambda_i) e^{-\imath n_i \theta_j} \right| \le \left(3 + \frac{\pi}{2}\log q\right)\gamma,
$$
for $i \in \{ 1, \ldots , m\}$, and $j\in \{1, \ldots , p\}$.
\end{lemma}
\begin{proof}
We can write
\begin{equation}
\begin{split}
\sum_{k=0}^{q-1} \frac{\imath^{n_i}}{s}
\sum_{\ell = 0 }^{s-1} e^{-\imath x_j \cdot \xi_{k \ell}}  e^{- \imath n_i
\phi_\ell} u_k(\lambda_i) 
&= \sum_{k=0}^{q-1} \left( J_n(r_j t_k) e^{-\imath n_i \theta_j} + \delta_{k i j}^\text{ang}  \right) u_k(\lambda_i)  \\
&= J_n(r_j \lambda_i) e^{-\imath n_i \theta_j} + \delta_{ij}^\text{rad} + \sum_{k=0}^{q-1} \delta_{k i j}^\text{ang} u_k(\lambda_i),
\end{split}
\end{equation}
where $\delta_{k i j}^\text{ang}$ and $\delta^\text{rad}_{ij}$ are the errors from discretizing the angles and using interpolation in the radial direction, respectively. 
By Lemma \ref{lem:num_angular_nodes} and  \ref{lem:num_radial_nodes} it follows that the error satisfies
$$
|\delta_{ij}^\text{rad} + \sum_{k=0}^{q-1} \delta_{k i j}^\text{ang} u_k(\lambda_i)| \le \gamma + \gamma(2 + \frac{\pi}{2} \log q ),
$$
which completes the proof.
\end{proof}

We will use the following property of Chebyshev interpolation polynomials; see \cite[Eq. 11]{rokhlin1988fast}.
\begin{lemma}\label{lemma:chebyshev_perturbation}
Let $t_k$ be Chebyshev nodes of the first kind defined in \eqref{eq:def_chebyshev} for the interval $[\lambda_1,\lambda_m]$. Then,
$$
\sum_{k=0}^{q-1} |u_k(t)| \le 2 + \frac{2}{\pi}\log q,
\quad \text{where} \quad
u_k(t) = \frac{\prod_{\ell \neq k} (t-t_\ell)}{\prod_{\ell \neq k} (t_k-t_\ell)}, 
$$
for all $t \in [\lambda_1,\lambda_m]$.
\end{lemma}

We will also require a classical result on discretization errors for integrals of smooth periodic functions.
\begin{lemma} \label{newlemma}
 Suppose that $g : [0,2\pi] \rightarrow \mathbb{C}$ is a smooth periodic 
function on the torus $[0,2\pi]$ where $0$ and $2\pi$ are identified. Then
\begin{equation}\label{eq:periodic_int_error}
     \left| \int_0^{2\pi} g(\phi) d \phi - \frac{2\pi}{s} \sum_{\ell=0}^{s-1} g(\phi_\ell) \right|
 < 4 \frac{\|g^{(s)}\|_{L^1}}{s^s},
\end{equation}
for all $s \ge 2$,
where $\phi_\ell = 2\pi \ell /s$, where $g^{(s)}(\phi)$ denotes the $s$-th derivative of $g(\phi)$ with respect to $\phi$.
\end{lemma}
See \cite[Theorem $1.1$]{kurganov2009order} for a proof. Lastly, we prove an upper bound on the normalization constants $c_{nk}$.
\begin{lemma}\label{lem:cj}
If $\lambda_{nk} \le \sqrt{\pi p}$, then the constants $c_{nk}$ satisfy $|c_{nk}| \le \sqrt{2 p}$.
\end{lemma}
\begin{proof}[Proof of Lemma \ref{lem:cj}]
We start with an alternate equivalent definition to \eqref{eq:eigenfun_const}:
\begin{equation}
        c_{nk} =
        \frac{1}{|\pi^{1/2} J_{n}'(\lambda_{n k})|}, 
        \quad \text{for} \quad (n,k) \in \mathbb{Z} \times \mathbb{Z}_{>0},
\end{equation}
see \cite[Eq. 10.6.3, Eq. 10.22.37]{dlmf}. 
By \cite[10.18.4, 10.18.6]{dlmf}
$$
J_n(x) = M_n(x) \cos( \theta_n(x)),
$$
where $M_n(x)^2 = J_n(x)^2 + Y_n(x)^2$ is a magnitude function, $Y_n$ is the $n$-th order Bessel function of the second kind, and $\theta_n(x)$ is a phase function. Taking the derivative gives
$$
J_n'(x) = M_n'(x) \cos(\theta_n(x)) - M_n(x) \sin(\theta_n(x)) \theta_n'(x).
$$
By \cite[8.479]{gradshteyn2014table}, we have
\begin{equation} \label{greq}
\frac{\pi }{2 \sqrt{x^2 - n^2}} \ge M_n(x)^2 \ge \frac{\pi }{2 x}.
\end{equation}
In particular, the magnitude function $M_n(x)$ does not vanish, so at a root $\lambda_{nk}$ of $J_n$, we must have $\theta_n(\lambda_{nk}) = \frac{\pi}{2} + \pi \ell$ for $\ell \in \mathbb{Z}$. It follows that
$$
J_n'(\lambda_{nk})^2 = M_n(\lambda_{nk})^2 \theta_n'(\lambda_{nk})^2.
$$
Using \cite[10.18.17]{dlmf} and \eqref{greq} gives
$$
J_n'(\lambda_{nk})^2 = \left( \frac{2}{\pi \lambda_{nk}} \right)^2 M_n(\lambda_{nk})^{-2}\ge \left(\frac{2}{\pi \lambda_{nk}} \right)^2 \frac{2 \sqrt{\lambda_{nk}^2 - n^2}}{\pi}.
$$
By \cite[Eq. 1.6]{elbert2001some} we have $\lambda_{nk} > n+ k \pi - \pi/2 + 1/2 > n+2$ for $(n,k) \in \mathbb{Z}_{\ge 0} \times \mathbb{Z}_{>0}$, which implies
$\sqrt{\lambda_{nk}^2 - n^2} \ge 2$ (this bound can be refined but is sufficient for the purpose of proving this lemma). Using this inequality together with the fact that $2 (2/\pi)^3 \ge 1/2$ gives
$$
c_{nk} = \frac{1}{\pi^{1/2} |J_{n}'(\lambda_{n k})|} \le \frac{2^{1/2} \lambda_{nk}}{\pi^{1/2}} \le\sqrt{2 p},
$$
where the final inequality follows from the assumption $\lambda_{nk} \le \sqrt{\pi p}$.
\end{proof}

\end{document}